\documentclass{siamart190516}

\usepackage{graphicx} 
\usepackage{amsfonts,verbatim,cite,subcaption,tikz,float,multicol,url,booktabs,mathabx}
\usepackage[section]{placeins}
\usepackage[algo2e,ruled,linesnumbered]{algorithm2e}

\newcommand{\vs}{\textit{vs}}
\newcommand{\argmin}{\mathop{\arg\min}}

\newcommand{\Ttran}{\mathsf{T}}
\newcommand{\T}{\mathsf{T}}

\newcommand{\tol}{\mathsf{tol}}

\newcommand{\Tr}{\mathsf{Tr}}
\newcommand{\Rq}{\mathsf{Rq}}

\newcommand{\defi}{:=}

\newcommand{\normml}{\left\vert\kern-0.25ex\left\vert\kern-0.25ex\left\vert}  
\newcommand{\normmr}{\right\vert\kern-0.25ex\right\vert\kern-0.25ex\right\vert}
\newcommand{\order}{\mathcal{O}}
\newcommand{\fl}{\mathsf{fl}}
\newcommand{\lan}{\mathsf{Lan}}
\newcommand{\ks}{\mathsf{KS}}
\newcommand{\iter}{\mathsf{Iter}}
\newcommand{\lc}{\mathsf{LC}}
\newcommand{\comp}{\mathsf{Ortho}}

\newcommand{\compress}{N_\mathsf{comp}}

\newcommand{\R}{\mathbb{R}}

\newcommand{\range}{\mathsf{range}}

\newcommand{\gap}{\mathsf{relgap}}

\newcommand{\mpr}{\mathbf{u}}
\newcommand{\fig}{eps}

\newcommand{\figsizeD}{0.43\textwidth}
\newcommand{\figsizeT}{0.3\textwidth}
\newcommand{\figsizeF}{0.24\textwidth}

\providecommand{\spa}[1]{\mathsf{span}\{#1\}}

\providecommand{\diagm}{\mathrm{diag}}
\providecommand{\diagM}[1]{\mathrm{diag}(#1)}
\providecommand{\abs}[1]{\lvert#1\rvert}
\providecommand{\norm}[1]{\lVert#1\rVert}

\providecommand{\bignorm}[1]{\bigl\lVert#1\bigr\rVert}
\providecommand{\Bignorm}[1]{\Bigl\lVert#1\Bigr\rVert}

\newtheorem{remark}[theorem]{Remark}

\providecommand{\edit}[1]{\textcolor{black}{#1}}

\crefname{equation}{}{}
\crefname{Assumption}{Assumption}{Assumptions}
\crefname{algocf}{Algorithm}{Algorithms}

\begin{document}

\title{Lanczos with compression for symmetric eigenvalue problems}
 \author{
     Angelo A. Casulli\thanks{\href{https://orcid.org/0000-0003-4068-5502}{https://orcid.org/0000-0003-4068-5502}, 
         (\href{mailto:angelo.casulli@gssi.it}{angelo.casulli@gssi.it}).
     }\and
     Daniel Kressner\thanks{
         Institute of Mathematics, EPFL, Lausanne, Switzerland
         (\href{mailto:daniel.kressner@epfl.ch}{daniel.kressner@epfl.ch}, \href{mailto:nian.shao@epfl.ch}{nian.shao@epfl.ch}).
     }
     \and
     Nian Shao$^\dagger$
 }
 \headers{Lanczos with compression for Symmetric EVP}{Angelo A. Casulli,  Daniel Kressner, Nian Shao}

\maketitle
\begin{abstract}
    The Lanczos method with implicit restarting is one of the most popular methods for finding a few exterior eigenpairs of a large symmetric matrix $A$. Usually based on polynomial filtering, restarting is crucial to limit memory and the cost of orthogonalization. 
    In this work, we propose a novel strategy for the same purpose, called Lanczos with compression. 
    Unlike polynomial filtering, our approach compresses the Krylov subspace using rational approximation and, in doing so, it sacrifices the structure of the associated Krylov decomposition. Nevertheless, it remains compatible with subsequent Lanczos steps and the overall algorithm is still solely based on matrix-vector products with $A$.
    On the theoretical side, we show that compression introduces only a small error compared to standard (unrestarted) Lanczos and therefore has only a negligible impact on convergence. Comparable guarantees are not available for commonly used implicit restarting strategies, including the Krylov--Schur method. On the practical side, our 
    numerical experiments demonstrate \edit{the potential of compression to} outperform the Krylov--Schur method in terms of matrix-vector products.
\end{abstract}

\section{Introduction}

This work is concerned with computing a few smallest or largest eigenvalues and the corresponding eigenvectors of a large symmetric matrix~$A$. 

Krylov subspace methods are among the most effective approaches for solving large-scale symmetric and nonsymmetric eigenvalue problems. Given an initial vector $q_1$, the $i$th Krylov subspace is generated by repeatedly multiplying a general square matrix $A$ with $q_1$:
\begin{equation*}
    \mathcal{K}_{i}(A,q_1) \defi \spa{q_1,Aq_1,\dotsc,A^{i-1}q_1}.
\end{equation*}
The Arnoldi process~\cite[Sec. 10.5]{Golub2013} is used to construct an orthonormal basis $Q_{i}$ for $\mathcal{K}_{i}(A,q_1)$. In turn, the Rayleigh--Ritz process is used for extracting eigenvalue and eigenvector approximations from $Q_i$, referred to as Ritz values and Ritz vectors.
A significant challenge of Krylov subspace methods arises from the need for storing $Q_i$. In the case of slow convergence (with respect to $i$), the available memory  may be exhausted long before achieving satisfactory approximation quality. Moreover, the cost of orthogonalizing $Q_i$ typically grows quadratically with $i$. 
Popular algorithms for solving large-scale eigenvalue problems, such as Sorensen's Implicitly Restarted Arnoldi (IRA) method~\cite{sorensen1992implicit},
address this challenge by combining the Arnoldi process with \emph{restarting}.
In this work, we propose \emph{compression} of Krylov subspaces as an alternative to restarting.
Explicit restarting uses the information gained from one cycle of the Arnoldi process to design an updated initial vector $p(A) q_1$ for the next Arnoldi cycle, with a (low-degree) polynomial $p$ that aims at filtering out the unwanted part of the spectrum. IRA is mathematically equivalent to explicit restarting, but it forms $p(A) q_1$ implicitly and performs several Arnoldi steps, without requiring any extra matrix-vector products.
Since its release in the late 1990s, the ARPACK software~\cite{lehoucq1998arpack} based on IRA has been one of the most widely used large-scale eigenvalue solvers. ARPACK is the basis of 
SciPy's and Julia's {\tt eigs} functions.
Often, the filter polynomial $p$ is chosen such its roots coincide with the unwanted Ritz values obtained after one Arnoldi cycle, a choice sometimes referred to as ``exact shifts''. In this case, IRA becomes mathematically equivalent to Stewart's Krylov--Schur method~\cite{stewart2002krylov}, which is simpler and numerically more reliable to implement than IRA, because it employs reordered Schur decompositions (instead of implicit QR steps) to perform restarts and deflations. The Krylov--Schur method is the basis of Matlab's 
{\tt eigs} function and it is implemented in SLEPc~\cite{hernandez2007parallel}
as well as Trillinos' Anasazi software~\cite{Baker2009}.

When $A$ is symmetric, the Arnoldi process simplifies to the Lanczos process~~\cite[Sec. 10.5]{Golub2013}. Restarting also simplifies; the IRA and Krylov--Schur methods become the IRL~\cite{calvetti1994implicitly} and thick-restarted Lanczos methods~\cite{wu2000thick}, respectively. During the last two decades, there
have been numerous further developments of restarted Lanczos methods, including extensions to block Lanczos methods~\cite{zhou2008block, baglama2003irbl}, adaptive strategies for selecting the number of Ritz vectors retained after each
restart~\cite{yamazaki2010adaptive}, strategies for choosing polynomial filters to compute interior eigenvalues~\cite{li2016thick, baglama1996iterative}, and communication-avoiding implementations for distributed-memory systems \cite{yamazaki2012communication}. In all these developments, the fundamental principle of using polynomial filtering for restarting has remained the same. 

From a theoretical perspective, the convergence analysis of (restarted) Arnoldi methods is by no means nontrivial. In the symmetric case, the convergence properties of the standard Lanczos method are relatively well understood; see, e.g.,~\cite{Kaniel1966,Saad1980, Meyer2024}. The convergence of restarted Arnoldi methods has been linked to polynomial approximation problems in~\cite{Beattie2005}, which allows to quantify convergence rates when using a fixed (well designed) filter polynomial $p$. In the case of exact shifts, the filter polynomial changes in every cycle and existing convergence results~\cite{aishima2015global,sorensen1992implicit} are of a qualitative nature, establishing global convergence without providing specific convergence rates.

As an alternative to restarting, this work proposes a novel compression strategy to limit the memory requirements of the Lanczos method for a symmetric matrix $A$. To provide some intuition behind our approach, suppose that the spectrum of $A$ is ordered such that
\begin{equation} \label{eq:ordereigenvalues}
\lambda_{1} < \dotsb < \lambda_{k} < \tau < 
\lambda_{k+1} \leq  \dotsb \leq  \lambda_{n},
\end{equation}
with a shift $\tau$ separating the smallest $k \ll n$ eigenvalues to be computed from the other eigenvalues. 
We first introduce a more general Krylov-like decomposition in \cref{sec:restart}, allowing us to compress the Lanczos process by an arbitrary orthonormal matrix.
Unlike restarting, compression sacrifices the Krylov structure but remains compatible with subsequent Lanczos steps.
Letting $U_{k}$ be an orthonormal basis of the eigenspace associated with the smallest $k$ eigenvalues, the best (but unrealistic) compression of an orthonormal basis $Q_{m}$ for $\mathcal{K}_{m}(A,q_{1})$ is given by $\range(Q_{m}Q_{m}^{\Ttran}U_{k})$, as it preserves all information relevant for (the unknown) $U_k$. 
Let $\chi_{\tau}(x)$ denote the step function 
\begin{equation*}
    \chi_{\tau}(x):= \begin{cases}
                  1 & \text{if } x < \tau, \\
                  0 & \text{otherwise.}
                 \end{cases}
\end{equation*}
As explained in \cref{sec:Vell}, a rational approximation to $\chi_{\tau}$ can be used to construct an orthonormal matrix that compresses $\mathcal{K}_{m}(A,q_{1})$ into a lower-dimensional subspace, while preserving nearly all essential information.
This parallels the developments from~\cite{casulli2024low}, which uses a rational approximation to a general function $f$ to compress the Lanczos process for approximating a matrix function $f(A)b$. To make the compression of the Lanczos method for eigenvalue problems practical, several additional ingredients are needed. In particular, reorthogonalization is less of a concern for matrix functions but of utmost importance for eigenvalue problems to, e.g., avoid so-called ghost eigenvalues.
As we will see in~\cref{sec:reorth}, full reorthogonalization alone is not sufficient to ensure stability in Lanczos with compression. To address this issue, we propose a new strategy called reorthogonalization with fill-in and establish a backward stability result, similar to the Krylov--Schur method in \cite{stewart2002krylov}.
In \cref{sec:convergence}, the convergence of our Lanczos with compression is analyzed. This is achieved by splitting the error between the Ritz values and the eigenvalues into two components: (1) the approximation error of standard Lanczos (without compression), and (2) the additional error incurred by compression. 
A back-of-the-envelope complexity analysis of the (re)orthogonalization cost for Lanczos with compression is also provided, and it reveals a clear advantage for problems that feature small relative eigenvalue gaps. 
In \cref{sec:numericalresults}, we report a series of numerical experiments involving matrices from a diversity of applications, \edit{illustrating the potential of our approach to outperform the Krylov--Schur method.} 

\begin{paragraph}{Notation} Throughout the paper, $\|\cdot\|$ denotes the Euclidean norm of a vector or the spectral norm of a matrix. An $n\times m$ matrix $Q$ is called orthonormal if its columns form an orthonormal basis, that is, $Q^\Ttran Q = I_m$ with the $m\times m$ identity matrix $I_m$.
\end{paragraph}

\section{Algorithm description}

In this section, we will derive our newly proposed Lanczos method with compression. For this purpose, 
 we will assume that the eigenvalues of the symmetric matrix $A\in\R^{n\times n}$ satisfy the (gapped) increased ordering~\eqref{eq:ordereigenvalues} and we aim at computing the smallest $k \ll n$ eigenvalues $\lambda_1,\ldots,\lambda_k$ as well as their associated eigenvectors. If the largest eigenvalues are of interest, one can simply replace $A$ by $-A$.

\subsection{Lanczos process and Krylov--Schur restarting}

We start by recalling the basic Lanczos process.
Given an initial vector $q_{1}$ with $\|q_1\| = 1$, $\lan \in \mathbb N$ iterations of the Lanczos process generate an orthonormal basis $Q_{\lan}$ for the Krylov subspace $\mathcal{K}_{\lan}(A,q_{1})$. More specifically, the Lanczos process returns the coefficients of the so-called \emph{Lanczos decomposition}:
\begin{equation}
    \label{eq:lanD}
    AQ_{\lan}
    = Q_{\lan}T_{\lan}+\beta_{\lan}q_{\lan+1}e_{\lan}^{\Ttran},
\end{equation}
where $\beta_{\lan} \in \R$, $Q_{\lan} = \begin{bmatrix}
        q_{1},\dotsc,q_{\lan}
    \end{bmatrix} \in \R^{n\times \lan}$, 
\begin{equation}
    \label{eq:lanD1}
    T_{\lan} = \begin{bmatrix}
        \alpha_{1} & \beta_{1} & &\\ 
        \beta_{1} & \ddots &\ddots&\\ 
        &\ddots &\ddots &\beta_{\lan-1}\\
        &&\beta_{\lan-1} &\alpha_{\lan} 
    \end{bmatrix}, \quad
    e_{\lan} = \begin{bmatrix}
    0 \\
    \vdots \\
    0 \\
    1
    \end{bmatrix} \in \R^{\lan}.
\end{equation}
Here and in the following, we will ignore the rare (fortunate) event that the Lanczos process breaks down within the first $\lan$ steps~\cite[Sec. 10.1.3]{Golub2013}. The vector $q_{\lan+1} \in \R^{n}$ is orthogonal to $q_{1},\dotsc,q_{\lan}$. In turn, it holds that $T_{\lan} = Q_{\lan}^{\Ttran}AQ_{\lan}$.
Thanks to the tridiagonal structure of $T_{\lan}$, it can be easily seen that the basis vectors admit the recurrence
\begin{equation}
    \label{eq:3recur}
    q_{i+1} = \frac{1}{\beta_{i}}\bigl(Aq_{i}-\alpha_{i}q_{i}-\beta_{i-1}q_{i-1}\bigr), \quad i = \edit{1,}2,\ldots,\lan.
\end{equation} 

In principle, the three-term recurrence~\eqref{eq:3recur} allows one to carry out the entire Lanczos process while storing only two additional vectors. However, in finite-precision arithmetic,  roundoff error incurs a significant loss of orthogonality, and repeated reorthogonalization is needed to ensure numerical orthogonality of $Q_{\lan}$. Reorthogonalization is commonly performed by orthogonalizing a freshly computed vector $q_{i+1}$ to \emph{all} previously computed vectors $q_1,\ldots,q_i$. In turn, reorthogonalization requires the storage of the entire basis $Q_{\lan}$ and an additional $\order(\lan^{2}\cdot n)$ operations.

As explained in the introduction, restarting is the standard approach to deal with the growing computational cost of the Lanczos process with reorthogonalization as $\lan$ increases. Krylov--Schur / thick restarting proceed by computing a spectral decomposition of $T_{\lan}$ and truncating it to the partial decomposition \[ T_{\lan}V_{\ks} = V_{\ks}\Lambda_{\ks}.\] Here, the diagonal matrix $\Lambda_{\ks}$ contains the smallest $\ks$ Ritz values, for some integer $\ks < \lan$, and $V_{\ks} \in \R^{\lan\times \ks}$ contains the corresponding (orthonormal) eigenvectors of $T_{\lan}$. 
Multiplication with $V_{\ks}$ transforms the Lanczos decomposition~\eqref{eq:lanD} into a smaller, so-called \emph{Krylov decomposition}:
\begin{equation}
    \label{eq:KrylovD}
    AQ_{\ks} = Q_{\ks}\Lambda_{\ks}+q_{\lan+1}b_{\ks}^{\Ttran},
\end{equation}
where $q_{\lan+1}$ is from the Lanczos decomposition \cref{eq:lanD},
$Q_{\ks} = Q_{\lan}V_{\ks}$, and $b_{\ks} = \beta_{\lan}V_{\ks}^{\Ttran}e_{\lan}$. 
By suitably adjusting the Lanczos process, one expands the basis $Q_{\ks}$ again to size $\lan$. One notable difference is that this destroys the tridiagonal structure of $T_{\lan}$ because the corresponding matrix contains the (usually full) vector $b_{\ks}$ in row and column $\ks+1$.

\subsection{Compression mechanism}
\label{sec:restart}

We will now describe a compression mechanism that is significantly more general than Krylov--Schur restarting. For this purpose,
we need to introduce the following, more general Krylov-like decomposition:
\begin{equation}
    \label{eq:LCD}
    AQ_{m}=Q_{m}T_{m}+q_{i+1}b_{m}^{\Ttran}+F_{m}
    \quad\text{and}\quad 
    Q_{m+1}^{\Ttran}F_{m}=0,
\end{equation}
where $Q_{m+1}\defi [Q_{m},q_{i+1}]$ has orthonormal columns, $T_{m}$ is symmetric (not necessarily tridiagonal) and $b_{m} \in \R^m$ is a general vector (not necessarily a scalar multiple of $e_m$).
Like in the Krylov decomposition \cref{eq:KrylovD}, $q_{i+1}$ is a vector from Lanczos process \cref{eq:3recur} with $i\geq m$. 
In contrast to existing Krylov decompositions, like~\eqref{eq:lanD} and~\eqref{eq:KrylovD}, we allow for the extra term $F_{m}$.
The condition $Q_{m+1}^{\Ttran}F_{m}=0$ ensures
\begin{equation*}
    Q_{m}^{\Ttran}AQ_{m} = Q_{m}^{\Ttran} Q_{m} T_{m} +Q_{m}^{\Ttran}q_{i+1}b_{m}^{\Ttran}+Q_{m}^{\Ttran}F_{m} = T_m.
\end{equation*} 
Hence, one can still extract Ritz values/vectors from the spectral decomposition of $T_m$.
The following theorem demonstrates a key advantage of~\eqref{eq:LCD}: Compressing $Q_{m}$ to $\ell < m$ columns by multiplying with an \emph{arbitrary} orthonormal matrix $V_{\ell}$ preserves the structure of \cref{eq:LCD}. 
\begin{theorem}
    \label{thmLan}
    Consider the Krylov-like decomposition \cref{eq:LCD} for a symmetric matrix $A\in\R^{n\times n}$.
    Given an orthonormal matrix $V_{\ell} \in \R^{m \times \ell}$, with $\ell< m$,
    set $\widetilde{Q}_{\ell} := Q_{m}V_{\ell} \in \R^{n\times \ell}$. Then
    \begin{equation*}
        A\widetilde{Q}_{\ell}=\widetilde{Q}_{\ell}
        \widetilde{T}_{\ell}
        +q_{i+1}\widetilde{b}_{\ell}^{\Ttran}
        +\widetilde{F}_{\ell} \quad \text{and}\quad 
        \widetilde{Q}_{\ell+1}^{\Ttran}\widetilde{F}_{\ell}=0,
    \end{equation*}
    where
    \begin{equation*}
        \widetilde{Q}_{\ell+1}=[\widetilde{Q}_{\ell},q_{i+1}], \quad \widetilde{T}_{\ell} = V_{\ell}^{\Ttran}T_{m}V_{\ell},\quad
        \widetilde{b}_{\ell}=V_{\ell}^{\Ttran}b_{m}, \quad \widetilde{F}_{\ell} = Q_{m}(I-V_{\ell}V_{\ell}^{\Ttran})T_{m}V_{\ell}+F_{m}V_{\ell}.
    \end{equation*} 
\end{theorem}
\begin{proof}
    By multiplying the decomposition~\cref{eq:LCD} with $V_\ell$, we obtain
    \begin{equation*}
        \begin{aligned}
            A\widetilde{Q}_{\ell} &= AQ_{m}V_{\ell}=Q_{m}T_{m}V_{\ell}+q_{i+1}b_{m}^{\Ttran}V_{\ell}+F_{m}V_{\ell}\\ 
            &=Q_{m}V_{\ell}(V_{\ell}^{\Ttran}T_{m}V_{\ell})+q_{i+1}(V_{\ell}^{\Ttran}b_{m})^{\Ttran}+Q_{m}(I-V_{\ell}V_{\ell}^{\Ttran})T_{m}V_{\ell}+F_{m}V_{\ell}\\ 
            &=\widetilde{Q}_{\ell}\widetilde{T}_{\ell}+q_{i+1}\widetilde{b}_{\ell}^{\Ttran}+\widetilde{F}_{\ell},
        \end{aligned}
    \end{equation*}
    where 
    \begin{equation*}
        \widetilde{Q}_{\ell}^{\Ttran}\widetilde{F}_{\ell} =
            V_{\ell}^{\Ttran}Q_{m}^{\Ttran}\bigl(Q_{m}(I-V_{\ell}V_{\ell}^{\Ttran})T_{m}V_{\ell}+F_{m}V_{\ell}\bigr)=0
            \quad\text{and}\quad q_{i+1}^{\Ttran}\widetilde{F}_{\ell}=0.
    \end{equation*}
\end{proof}

\edit{Note that the algorithms developed in this paper never form or store $F_m$ explicitly. In fact,
\eqref{eq:LCD} can equivalently be characterized without introducing $F_m$:
the difference between $AQ_{m}$ and $Q_{m}T_{m}+q_{i+1}b_{m}^{\Ttran}$ is orthogonal to the range of $Q_{m+1}$.
Also,}
note that \cref{thmLan} contains Krylov--Schur restarting as a special case; if $F_{m} = 0$ and $V_\ell$ consists of eigenvectors of $T_m$ then $\widetilde{F}_\ell = 0$. 

Given a Krylov-like decomposition~\cref{eq:LCD}, the Lanczos process can be continued in a manner similar to the Krylov--Schur method \cite[Eq.~(2.3)]{wu2000thick}, yielding the next Lanczos vector
\begin{equation}
    \label{eq:qk2}
    q_{i+2} \defi \beta_{i+1}^{-1}(I-Q_{m+1}Q_{m+1}^{\Ttran})Aq_{i+1},    
\end{equation}
where $\beta_{i+1}$ is a normalization factor such that $\norm{q_{i+2}}=1$. The same procedure can be applied to continue the 
compressed decomposition described in \cref{thmLan}. Specifically, we first set $\widetilde{q}_{\ell+1}=q_{i+1}$, then compute $A\widetilde{q}_{\ell+1}$ and orthogonalize it with respect to $\widetilde{Q}_{\ell+1}$:
\begin{equation}
    \label{eq:ql2}
    \widetilde{q}_{\ell+2} \defi \widetilde{\beta}_{\ell+1}^{-1}(I-\widetilde{Q}_{\ell+1}\widetilde{Q}_{\ell+1}^{\Ttran})A\widetilde{q}_{\ell+1},    
\end{equation}
where $\widetilde{\beta}_{\ell+1}$ ensures $\norm{\widetilde{q}_{\ell+2}}=1$. With $\widetilde{q}_{\ell+1}$ and $\widetilde{q}_{\ell+2}$ being defined, the Lanczos process now proceeds in the usual manner; via the three-term recurrence~\cref{eq:3recur}. 
From a theoretical and practical point of view it is important that the compression does \emph{not} affect the future vectors generated by the Lanczos process, which is ensured when $\widetilde{q}_{\ell+2}=q_{i+2}$. The following proposition shows that this is indeed the case under the additional, relatively mild condition $b_{m}\in \range(V_{\ell})$. Interestingly, standard Krylov--Schur restarting generally does not satisfy this condition and does not preserve the subsequent Lanczos process\edit{, that is, $\widetilde{q}_{\ell+2}\neq q_{i+2}$}.

\begin{proposition}
    \label{prop:compression}
    Given the Krylov-like decomposition \cref{eq:LCD}, let $q_{i+2}$ and $\widetilde{q}_{\ell+2}$ denote the next Lanczos vectors obtained without and with compression using $V_{\ell}$, as defined in \cref{eq:qk2} and \cref{eq:ql2}, respectively. Assume that $V_{\ell}V_{\ell}^{\Ttran}b_{m}=b_{m}$. Then, up to a sign change, $\widetilde{\beta}_{\ell+1}=\beta_{i+1}$ and $\widetilde{q}_{\ell+2} = q_{i+2}$.
\end{proposition}
\begin{proof}
    Multiplying $q_{i+1}^{\Ttran}$ from the left to \cref{eq:LCD} gives 
    $b_{m}=Q_{m}^{\Ttran}Aq_{i+1}$. Because of $V_{\ell}V_{\ell}^{\Ttran}b_{m}=b_{m}$, we have
    \begin{equation*}
        Q_{m}Q_{m}^{\Ttran}Aq_{i+1} = Q_{m}b_{m} = Q_{m}V_{\ell}V_{\ell}^{\Ttran}b_{m} = Q_{m}V_{\ell}V_{\ell}^{\Ttran}Q_{m}^{\Ttran}Aq_{i+1} = \widetilde{Q}_{\ell}\widetilde{Q}_{\ell}^{\Ttran}Aq_{i+1}.
    \end{equation*}
    In turn,
    \begin{equation*}
        \beta_{i+1}q_{i+2} = (I-Q_{m}Q_{m}^{\Ttran}-q_{i+1}q_{i+1}^{\Ttran})Aq_{i+1}  
        = (I-\widetilde{Q}_{\ell}\widetilde{Q}_{\ell}^{\Ttran}-q_{i+1}q_{i+1}^{\Ttran})Aq_{i+1} 
        =  \widetilde{\beta}_{\ell+1}\widetilde{q}_{\ell+2},
    \end{equation*}
    which implies the desired result.
\end{proof}

Our algorithm proceeds by cyclically alternating between expansion and compression. Specifically, the compressed basis ${Q}^{(1)}_{\ell} := \widetilde{Q}_{\ell}$, together with $q_{m+1}$, is expanded to an orthonormal basis of size $m+1$ by performing another $m-\ell$ steps of Lanczos process. 
The first $m$ vectors of the expanded basis are compressed again to a basis ${Q}^{(2)}_{\ell}$ of size $\ell$, expanded again, and so on, yielding the following procedure:
\begin{equation}
    \label{eq:compression}
    \begin{aligned}
        &q_{1}\\ 
        \text{Lanczos}\quad&\Downarrow\\ 
        \underline{q_{1},\dotsc,q_{m}},&q_{m+1}\\ 
        \text{compression}\quad&\Downarrow\\ 
        q_{1}^{(1)},\dotsc,q_{\ell}^{(1)},&q_{m+1}\\ 
        \text{Lanczos}\quad&\Downarrow\\ 
        \underline{q_{1}^{(1)},\dotsc,q_{\ell}^{(1)},q_{m+1},\dotsc,q_{2m-\ell}},&q_{2m-\ell+1}\\ 
        \text{compression}\quad&\Downarrow\\
        q_{1}^{(2)},\dotsc,q_{\ell}^{(2)},&q_{2m-\ell+1}\\ 
        \text{Lanczos}\quad&\Downarrow\\ 
        \underline{q_{1}^{(2)},\dotsc,q_{\ell}^{(2)},q_{2m-\ell+1},\dotsc,q_{3m-2\ell}},&q_{3m-2\ell+1}\\ 
        \text{compression}\quad&\Downarrow\\
        \cdots & \cdots
    \end{aligned}
\end{equation}

Each step of the procedure~\eqref{eq:compression} is associated with a Krylov-like decomposition~\cref{eq:LCD}.
After the first round of Lanczos, one obtains a Lanczos decomposition~\cref{eq:lanD} with $\lan = m$.
To satisfy the requirements of~\cref{prop:compression} and preserve subsequent Lanczos vectors, the vector $e_{m}$ needs to be contained in the range of the compression matrix $V_{\ell}^{(0)}$ that is used for performing the first compression. We will explain in \cref{sec:Vell} how this is achieved. 
By a change of basis, we may assume without loss of generality  
that $\big( V_{\ell}^{(0)}\big)^{\Ttran}e_{m}=e_{\ell}$.
According to~\cref{thmLan}, the Krylov-like decomposition after the first compression takes the form
\begin{equation}
    \label{eq:Tcompression}
    AQ_{\ell}^{(1)} = Q_{\ell}^{(1)}T_{\ell}^{(1)}+\beta_{m}q_{m+1}e_{\ell}^{\Ttran}+F_{\ell}^{(1)}
    \quad\text{and}\quad  \big[ Q_{\ell}^{(1)}, q_{m+1}  \big]^\Ttran F_{\ell}^{(1)} = 0,
\end{equation}
where $T^{(1)}_{\ell} = \big( V^{(0)}_{\ell} \big)^{\Ttran}T_{m}V^{(0)}_{\ell}$.
The second round of Lanczos process expands this into a Krylov-like decomposition of the form 
\begin{equation*}
    AQ_{m}^{(1)} = Q_{m}^{(1)}T_{m}^{(1)}+\beta_{2m-\ell}q_{2m-\ell+1}e_{m}^{\Ttran}+F_{m}^{(1)},
\end{equation*}
where 
\begin{equation*}
    T_{m}^{(1)} = \begin{bmatrix}
        T_{\ell}^{(1)} & \beta_{m}e_{\ell} & &\\ 
        \beta_{m}e_{\ell}^{\Ttran} & \alpha_{m+1} & \beta_{m+1} & \\ 
        & \beta_{m+1} & \ddots &\ddots& \\
        &  & \ddots &\ddots& \beta_{2m-1-\ell}\\
        && &\beta_{2m-1-\ell} & \alpha_{2m-\ell}
    \end{bmatrix}
    \quad\text{and}\quad 
    F_{m}^{(1)} = \begin{bmatrix}
        F_{\ell}^{(1)} \quad  0_{m-\ell}    
    \end{bmatrix}.
\end{equation*} 
This process of compressing and expanding Krylov-like decompositions repeats in an analogous fashion for the subsequent cycles of~\cref{eq:compression}.

\begin{remark}
    Given a Krylov-like decomposition~\cref{eq:LCD}, the range of $Q_{m}$ is a Krylov subspace for $A$ if and only if $(I-Q_{m}Q_{m}^{\Ttran})AQ_{m}=q_{i+1}b_{m}^{\Ttran}+F_{m}$ has rank one~\cite[Chap. 5, Thm. 1.1]{Stewart2001}.
    In our algorithm, compression yields a matrix $F_m$ that generally increases the rank and, hence, it destroys the Krylov structure of the subspace. Thus, compression fundamentally differs from (polynomial) restarting, which preserves the Krylov structure but alters the Lanczos sequence.
\end{remark}

\subsection{Construction of compression matrix $V_{\ell}$}
\label{sec:Vell}

We now turn to the choice of the orthonormal matrix that defines the compression mechanism described above. We begin by considering an idealized scenario in which compression preserves all relevant information. Building on techniques similar to those in \cite{casulli2024low}, we then demonstrate that a practical rational Krylov compression can retain most of the essential information.
For simplicity, we only discuss the first compression step $V_{\ell} \equiv V^{(0)}_{\ell}$; all subsequent compression steps are performed in an entirely analogous manner.

\begin{paragraph}{Idealized compression}
\emph{Ideally}, compression preserves
the components of the $k\le \ell$ (as yet unknown) eigenvectors of interest for $A$ in the currently computed basis. To specify such an idealized choice,
let $Q_{n}\in \R^{n \times n}$ be the orthogonal matrix obtained  after performing \emph{all} $n$ steps of the (uncompressed) Lanczos process applied to $A$, and define $T_{n}=Q_{n}^{\Ttran}AQ_{n}$.
Then $T_n$ has the same eigenvalues $\lambda_1,\ldots,\lambda_n$ as $A$. Letting $w_1,\ldots w_n$ denote an orthonormal basis of eigenvectors such that $T_{n} w_i = \lambda_i w_i$, we set
\[
 W_{k}=[w_{1},\dotsc,w_{k}].
\]
Note that $\range(Q_{n}W_{k})$ is precisely the subspace containing the $k$ desired  eigenvectors of $A$. We now consider $k< m\leq n$, and
recall that the first $m$ steps of the Lanczos process produce a basis $Q_{m}$ that constitutes the first $m$ columns of $Q_{n}$. This allows us to express
\[
 Q_{n}W_{k} = \begin{bmatrix}
               Q_m & \star\ 
              \end{bmatrix}
\begin{bmatrix}
               \overline{W}_{k} \\ \star
\end{bmatrix},
\]
where $\overline{W}_{k}\in\R^{m\times k}$ contains the first $m$ rows $W_{k}$.
In particular, the compression $Q_{m} \overline{W}_{k}$ of $Q_{m}$ perfectly preserves all information on the eigenvectors of interest.
This, of course, represents an unrealistic choice because $\overline{W}_{k}$ is unknown after the first $m$ steps.
\end{paragraph}

\begin{paragraph}{Injection of rational function}

We aim at using rational approximation to produce a slightly larger compression matrix $V_{\ell} \in \R^{m\times \ell}$, such that $\range(V_{\ell})$ contains the columns of the idealized choice $\overline{W}_{k}$ approximately. For this purpose, we will inject a real rational function $r = p/q$ with polynomials $p,q$. The (extended) list of poles of $r$ is denoted by
\[
 \Xi \defi \big[\xi_{1},\dotsc,\xi_{\abs{\Xi}}\big], \quad \xi_i \in \mathbb C \cup \{\infty\},
\]
where $\Xi$ contains the zeros of $q$ (including their multiplicities), supplemented by \edit{$\deg(p) - \deg(q)$} infinite poles if \edit{$\deg(p) > \deg(q)$.} The cardinality of $\Xi$ thus satisfies \edit{$\abs{\Xi} \defi \max\{\deg(p), \deg(q)\}+1$.} Given such a list of poles $\Xi$, we will make use of the 
rational Krylov subspace
\begin{equation}
    \mathcal{Q}(T_{m},e_m,\Xi)  \defi \bigl\{ r(T_m)e_m \colon \text{$r$ is a rational function with poles $\Xi$} \bigr\}   \label{eq:ratkrylov}
\end{equation}
for a general symmetric $m\times m$ matrix $T_m$.
As long as $\abs{\Xi} \le m$, the dimension of $\mathcal{Q}(T_{m},e_m,\Xi)$ is generically $\abs{\Xi}$, which will be assumed in the following.

The following lemma is crucial to our construction; it links matrix-vector products with $r(T_{n})$ for the (unknown) matrix $T_n$ to 
matrix-vector products with $r(T_{m})$ for a submatrix $T_m$ of $T_n$.
\begin{lemma}\label{thm:BKCS}
   Consider a symmetric $n\times n$
   matrix of the form
      \begin{equation*}
         T_{n}=\begin{bmatrix}
             T_{m}&\\& \star
         \end{bmatrix} + \beta \begin{bmatrix}
             &e_me_1^{\T}\\e_1e_m^{\T} &
         \end{bmatrix}\in \R^{n\times n},
     \end{equation*}
     with $T_{m} \in \R^{m \times m}$ and $\beta \in \R$.
     Let $r = p/q$ be a real rational function with poles $\Xi$ as defined above, and suppose that the eigenvalues
     of $T_m$ and $T_n$ are not zeros of $q$. Then, for any vector $v\in \R^{n}$, it holds that
     \begin{equation*}
         r(T_{n})  v \in
         \begin{bmatrix}
            r(T_{m})  \overline{v} + \mathcal{Q}(T_{m},e_m,\Xi) \\
              \R^{n-m}
         \end{bmatrix},
     \end{equation*}
     where $\overline{v}\in \R^m$ contains the first $m$ components of $v$, and $\mathcal{Q}(T_{m},e_m,\Xi)\subset \R^{m}$ denotes the rational Krylov subspace~\eqref{eq:ratkrylov}. 
\end{lemma}   
\begin{proof}
    The result of this lemma is an immediate consequence of \cite[Cor.~2.6]{casulli2024low}, which provides an explicit formula for $r(T_{n})v$.
\end{proof}
\begin{remark}
    Note that \cref{thm:BKCS} does not impose tridiagonal structure on the matrix $T_{m}$. Indeed, this structure is destroyed in subsequent compression steps, but \cref{thm:BKCS}, along with the strategy described below for determining the compression matrix $V_\ell$, remain applicable also after compression has been performed. 
\end{remark}

We come back to our goal of replacing the idealized compression matrix $\overline{W}_{k} = [ \overline{w}_{1}, \ldots, \overline{w}_{k} ]$, where each 
$\overline{w}_{i}$ contains the first $m$ components of the 
eigenvector $w_{i}$ of $T_{n}$. Applying \cref{thm:BKCS} with $v = w_i$ and additionally assuming that $r(\lambda_{i})\neq 0$, it follows that
\begin{equation} \label{eq:formulabarwi}
    \overline{w}_{i} \in \frac{1}{r(\lambda_{i})} r(T_{m})\overline{w}_{i} + \mathcal{Q}(T_{m},e_m,\Xi).
\end{equation}
This inclusion is not (yet) very helpful because the first term $r(T_{m})\overline{w}_{i}$ involves the unknown vector $\overline{w}_{i}$.
However, we will see that for particular choices of $r$, this term can be well approximated in a lower-dimensional subspace.
\end{paragraph}

\begin{paragraph}{Choice of rational function $r$}
 Letting $\theta_1 \leq \dotsb \leq \theta_m$ denote the eigenvalues of $T_{m}$ with associated normalized eigenvectors $s_{1},\dotsc,s_{m}$, suppose that
\begin{equation}
    \label{eqn:relaxed-condition_rational}
    \begin{aligned}
        \abs{r(\theta_i)-1}< \tol_{\mathrm{ra}}& \quad \text{for } i = 1, \dotsc, k, \\ 
        \abs{r(\theta_i)} < \tol_{\mathrm{ra}}& \quad \text{for }i = k+1, \dotsc, m,
    \end{aligned}         
\end{equation}
for some prescribed tolerance $0<\tol_{\mathrm{ra}} \ll 1$. Then $r(T_{m})$ is close to the spectral projector associated with the first $k$ eigenvectors of $T_{m}$ and, in turn, $r(T_{m})\overline{w}_{i}$ can be well approximated by a vector in \edit{the invariant subspace} $\spa{s_1, \dotsc, s_k}$.
Combined with~\eqref{eq:formulabarwi}, this implies that the range of 
$\overline{W}_{k}$ is well approximated by the subspace
\begin{equation}
    \label{eq:defVell}
    \range(V_{\ell}) = \spa{s_{1},\dotsc,s_{k}}+\mathcal{Q}(T_{m},e_m,\Xi),
\end{equation}
where $V_{\ell}$ denotes an orthonormal basis of that subspace with $\ell = k + \abs{\Xi}$.
In other words, we can compress $Q_{m}$ into $Q_{m}V_{\ell}$ without losing significant information.

To fix the list of poles $\Xi$ in the definition~\eqref{eq:defVell} of $V_{\ell}$, we need to
determine a rational function that satisfies~\cref{eqn:relaxed-condition_rational}. To attain a reasonably small approximation tolerance, this clearly requires $\theta_k < \theta_{k+1}$.
Let $\chi_{\tau}(x)$ be the step function for some 
$\tau \in (\theta_{k}, \theta_{k+1})$.
We now relax the discrete conditions~\cref{eqn:relaxed-condition_rational} into the continuous approximation problem of finding a rational function $r_\tau$ such that
\begin{equation}\label{eqn:rat_approx}
    |r_{\tau}(x) - \chi_{\tau}(x)| < \text{tol}_{\mathrm{ra}} \quad \text{for every} \quad x \in [\tau-\eta, \tau +\eta] \setminus (\tau - \delta, \tau + \delta),
\end{equation}
with $\eta>\delta>0$ chosen such that $\theta_1,\ldots,\theta_k \in [\tau-\eta,\tau-\delta]$ and $\theta_{k+1},\ldots,\theta_m \in [\tau+\delta, \tau+\eta]$.

To keep the dimension $\ell = k + \abs{\Xi}$ of the compressed subspace small, we aim at determining a low-degree rational function $r_\tau$.
To achieve this, we fix an upper bound $d$ (to be determined later) on the degree and solve the rational approximation problem
\begin{equation*}
r_{\tau} = \argmin_{\deg(r) \le d}
\max_{x \in [\tau - \eta, \tau + \eta] \setminus (\tau - \delta, \tau + \delta)}
\left| r(x) - \chi_{\tau}(x) \right|.
\end{equation*}
This approximation problem is known as Zolotarev's fourth problem, which has been extensively studied in the literature \cite[Sec.~4.3]{petrushev2011rational}.
In particular, explicit formulas for the best rational approximation $r_\tau$, expressed in terms of elliptic functions, are given in \cite[Thm.~4.8]{petrushev2011rational}, with the solution taking the form
\begin{equation} \label{eq:rtauform}
r_{\tau}(x) =\edit{\frac{1}{2}-} \frac{(x - \tau)\, p\bigl((x - \tau)^2\bigr)}{q\bigl((x - \tau)^2\bigr)},
\end{equation}
for certain polynomials $p$ and $q$ of the same degree.
Moreover, it is shown in \cite[Thm.~3.3]{beckermann2017singular} that:
\begin{equation*}
\max_{x \in [\tau - \eta, \tau + \eta] \setminus (\tau - \delta, \tau + \delta)}
\left| r_{\tau}(x) - \chi_{\tau}(x) \right| \leq 
\sqrt{Z_{d}},\quad \text{where } 
\edit{
Z_{d} \leq 
4 \Bigl[\exp\left(\frac{\pi^2}{2 \log(4\eta/\delta)}\right)\Bigr]^{-2d}.}
\end{equation*}
Therefore, choosing
\begin{equation}
    \label{eq:bound_degree}
d = \biggl\lceil \frac{2}{\pi^2}\log\Bigl(\frac{4}{\tol_{\mathrm{ra}}}\Bigr) \cdot  \log\Bigl(\frac{4\eta}{\delta}\Bigr) \biggr\rceil,
\end{equation}
the approximation error falls below the prescribed tolerance $\tol_{\mathrm{ra}}$ and the resulting rational function then satisfies \cref{eqn:rat_approx} as well as \cref{eqn:relaxed-condition_rational}.
In particular, the value of $d$ depends logarithmically on $\tol_{\mathrm{ra}}^{-1}$ and $\eta/\delta$.
Once the parameters $d$, $\tau$, $\delta$ and $\eta$ are fixed, we can compute the poles of $r_\tau$ using numerical approximations of elliptic functions, as described in \cite{Nakatsukasa2016}.

Note that the form~\cref{eq:rtauform} of $r_{\tau}(x)$ implies that its set of poles $\Xi$ always includes \edit{an infinite pole}, ensuring that
\begin{equation*}
    e_{m}\in \mathcal{Q}(T_{m},e_m,\Xi) \subset \range(V_{\ell}).
\end{equation*}
As a result, the corresponding assumption of~\cref{prop:compression} is always satisfied.

Let us set 
\begin{equation} \label{eq:stepfunset}
  \tau = (\theta_{k} + \theta_{k+1})/2, \quad \delta = (\theta_{k+1} - \theta_{k})/2, \quad \eta = \theta_{m} - \tau;
\end{equation}
see~\cref{fig:rat_approx} for an illustration. 
Since we aim at approximating only a few small eigenvalues, we expect $\tau$ to be closer to $\lambda_1$ than to $\theta_m$, especially when $m \gg k$. In particular, it is reasonable to assume that
\[
 \lambda_{1}>2\tau-\theta_{m}.
\]
By the variational characterization of $\lambda_{1}$, this implies $\theta_{i}\geq \lambda_{1}>2\tau-\theta_{m}=\tau-\eta$ for $i = 1,\ldots,k$. Thus,~\eqref{eq:stepfunset} is a feasible choice and the step function approximation~\eqref{eqn:rat_approx} with the parameters~\eqref{eq:stepfunset} implies 
that both conditions~\cref{eqn:relaxed-condition_rational} hold.
Moreover, by the Cauchy interlacing theorem~\cite[Thm.~8.1.7]{Golub2013}, $\lambda_{1}\leq \lambda_{i} \leq \theta_{k}$, implying that $r_{\tau}(\lambda_i) \neq 0$ for $i = 1,\ldots, k$.

 \begin{figure}[htbp]
    \centering
\begin{tikzpicture}
    \draw[thick] (-2,0) -- (9,0);
    \draw[thick] (0, -0.1) -- (0, 0.1);         
    \node[below] at (0,0) {$\theta_1$};

    \node[below] at (1,-.1) {$\dots$};
    
    \draw[thick] (2, -0.1) -- (2, 0.1);         
    \node[below] at (2,0) {$\theta_k$};

    \draw[thick] (3.5, -0.1) -- (3.5, 0.1);          
    \node[below] at (3.5,-0.1) {$\tau$};

    \draw[thick] (5, -0.1) -- (5, 0.1);  
    \node[below] at (5.2,0) {$\theta_{k+1}$};

    \node[below] at (6.5,-.1) {$\dots$};
    \draw[thick] (8, -0.1) -- (8, 0.1);   
    \node[below] at (8,0) {$\theta_m$};          \draw[<->, thick] (3.5, 0.5) -- (8, 0.5);    
    \node[above] at (5.75, 0.5) {$\eta$};            
    \draw[<->, thick] (-1, .5) -- (3.5, .5);        
    \node[above] at (1.25, .5) {$\eta$};              

      \draw[<->, thick] (3.5, .75) -- (5, .75);     
    \node[above] at (4, .75) {$\delta$};            

    \draw[<->, thick] (2, .75) -- (3.5, .75);         
    \node[above] at (3, .75) {$\delta$};

\end{tikzpicture}

\begin{tikzpicture}
    \draw[thick] (-2,0) -- (9,0);

    \draw[thick] (0, -0.1) -- (0, 0.1);          
    \node[below] at (0,0) {$\theta_1$};

    \node[below] at (1,-.1) {$\dots$};
    
    \draw[thick] (2, -0.1) -- (2, 0.1);          
    \node[below] at (2,0) {$\theta_k$};
    \node[below] at (2.45,-.1) {$\dots$};
    \node[below] at (3,0) {$\theta_{\widehat{k}}$};
    \draw[thick] (2.1, -0.06) -- (2.1, 0.06);
    \draw[thick] (2.2, -0.06) -- (2.2, 0.06);
    \draw[thick] (2.3, -0.06) -- (2.3, 0.06);
    \draw[thick] (2.4, -0.06) -- (2.4, 0.06);
    \draw[thick] (2.5, -0.06) -- (2.5, 0.06);
    \draw[thick] (2.6, -0.06) -- (2.6, 0.06);
    \draw[thick] (2.7, -0.06) -- (2.7, 0.06);
    \draw[thick] (2.8, -0.06) -- (2.8, 0.06); 
    \draw[thick] (2.9, -0.1) -- (2.9, 0.1);         
    \draw[thick] (3.5, -0.1) -- (3.5, 0.1);          
    \node[below] at (3.5,-0.1) {$\tau$};

    \draw[thick] (5, -0.1) -- (5, 0.1);    
    \node[below] at (5,0) {$\theta_{\widehat{k}+1}$};

    \node[below] at (6.5,-.1) {$\dots$};
    
    \draw[thick] (8, -0.1) -- (8, 0.1);   
    \node[below] at (8,0) {$\theta_m$};         \draw[<->, thick] (3.5, 0.5) -- (8, 0.5);     
    \node[above] at (5.75, 0.5) {$\eta$};            
    \draw[<->, thick] (-1, .5) -- (3.5, .5);        
    \node[above] at (1.25, .5) {$\eta$};              

      \draw[<->, thick] (3.5, .75) -- (5, .75);     
    \node[above] at (4, .75) {$\delta$};

    \draw[<->, thick] (2, .75) -- (3.5, .75);         
    \node[above] at (3, .75) {$\delta$};

\end{tikzpicture}

\caption{Illustration of the parameter choices~\cref{eq:stepfunset,eq:paraRA} for rational approximation.}
\label{fig:rat_approx}
\end{figure}
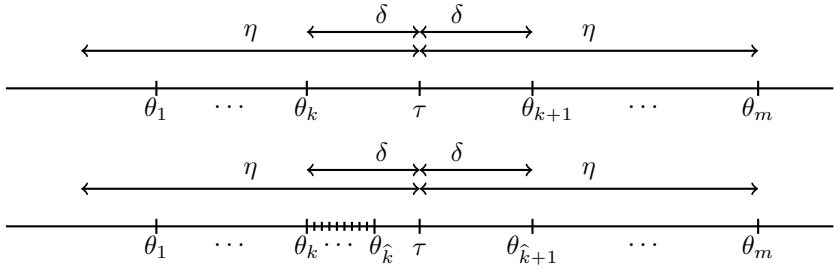
\end{paragraph}

\begin{paragraph}{Adaptive strategy to avoid tiny gaps}
While the dependence of the rational degree on the reciprocal relative gap $\eta/\delta$ is only logarithmic, a very small $\delta$, caused by 
$\theta_{k}$ and $\theta_{k+1}$ being very close, is still best avoided. It turns out that the gap can be easily increased by including additional eigenvectors in the basis.
Specifically, for $\widehat{k}$ with $k\leq \widehat{k}\leq m-1$, we let $V_{\ell,\widehat{k}}$ be an orthonormal basis that satisfies
\begin{equation}
    \label{eq:basisRA}
    \range(V_{\ell,\widehat{k}}) = \spa{s_{1},\dotsc,s_{\widehat{k}}}+\mathcal{Q}(T_{m},e_m,\Xi_{\widehat{k}}),
\end{equation}
where $\Xi_{\widehat{k}}$ is the of poles for a rational approximation $r_{\tau}$ satisfying \cref{eqn:rat_approx} with 
\begin{equation}
    \label{eq:paraRA}
    \tau = (\theta_k + \theta_{\widehat{k}+1}) / 2,\quad  \delta = (\theta_{\widehat{k}+1} - \theta_k) / 2 \quad\text{and}\quad \eta = \theta_m - \tau;
\end{equation}
see also~\cref{fig:rat_approx}.
\edit{Note that, as illustrated in the bottom panel of~\cref{fig:rat_approx}, the transition band $[\tau-\delta,\tau+\delta]$ may contain Ritz values $\theta_{k+1},\ldots,\theta_{\hat k}$. 
The rational approximation is not and does not need to be controlled at these Ritz values because the corresponding Ritz vectors are explicitly retained in the compressed basis.
}

Following the reasoning above, we expect $V_{\ell,\widehat{k}}$ \edit{from~\cref{eq:basisRA}} to be a good compression matrix.
While larger $\widehat{k}$ leads to larger gaps $\delta$ (and potentially reduced $\abs{\Xi_{\widehat{k}}}$), this comes at the cost of including more eigenvectors of $T_m$ in the compression matrix. \cref{alg:second_pole_strategy} strikes a compromise between these two aims by choosing $V_{\ell}$ to be the matrix $V_{\ell,\widehat{k}}$ with the minimal number $\ell$ of columns for $k\leq \widehat{k}\leq m-1$. We always assume that $V_{\ell}$ has fewer than $m$ columns; otherwise, $m$ should be increased further before attempting compression. \edit{In practice, we recommend $m=4k$.}

\color{black}
\begin{algorithm2e}[H]
    \caption{Adaptive strategy for selecting $V_{\ell}$}
    \label{alg:second_pole_strategy}
    \KwIn{Symmetric matrix $T_{m}$. Tolerance for rational approximation $\tol_{\mathrm{ra}}$. Integer $k$ for the number of desired eigenvalues.}
    \KwOut{Orthonormal basis $V_{\ell}$ for compression.}
Compute eigenpairs $(\theta_{1}, s_1), \dotsc, (\theta_{m}, s_m)$ of $T_m$\;
    \For{$\widehat{k}=k:m-1$}{
        Apply \cref{eq:bound_degree} to compute the degree $d_{\widehat{k}}$ of the rational approximation satisfying \cref{eqn:rat_approx} with parameters~\cref{eq:paraRA}\;
        }
   Find $k_{*}$ that minimizes $\widehat{k} + d_{\widehat{k}}$ and compute its poles $\Xi_{k_{*}}$\; 
   Return $V_{\ell}$ as the matrix $ V_{\ell,k_{*}}$ defined in \cref{eq:basisRA}\;
\end{algorithm2e}
In line~6 of \cref{alg:second_pole_strategy}, we need to construct the orthonormal basis $V_{\ell}$ as described in \cref{eq:basisRA}. 
For this purpose, we first apply the rational Arnoldi method from~\cite{berljafa2014} to construct an orthonormal basis for the rational Krylov subspace $\mathcal{Q}(T_{m},e_m,\Xi_{\widehat{k}})$. This is followed by computing the (economic) $Q$-factor of a (Householder) QR decomposition of the matrix consists of the first $\widehat{k}$ eigenvectors for $T_{m}$ and the rational Krylov subspace basis.
\end{paragraph}

\begin{paragraph}{Analysis of compression error}
 We conclude this section by showing that our choice of compression $V_{\ell}$ only introduces a small compression error in the relevant eigenvectors of $T_{n}$ and, hence, $A$.
\begin{theorem}
    \label{thm1}
    With $T_{n}$ defined as in \cref{thm:BKCS}, let
    $s_{1},\dotsc,s_{m} \in \R^m$ denote orthonormal eigenvectors 
    belonging to the eigenvalues  
    $\theta_{1}\leq \dotsb\leq \theta_{m}$ of $T_{m}$, and let 
    $W_k \in \R^{n \times k}$ contain 
    orthonormal eigenvectors belonging to the smallest $k$ eigenvalues $\lambda_1 \le \cdots \le \lambda_k$ of $T_n$.
    
    Given $0<\tol_{\mathrm{ra}}<1$ and $k\leq \widehat{k}\leq m-1$ such that $\theta_k < \theta_{\widehat{k}+1}$, set $\tau=(\theta_{k}+\theta_{\widehat{k}+1})/2$ and let $r_{\tau}$ be a rational function with poles $\Xi_{\widehat{k}}$ satisfying
    \begin{equation}
        \label{con:ratapp}
        \abs{r_{\tau}(x)-\chi_{\tau}(x)} <  \tol_{\mathrm{ra}} \quad \text{for every}\quad x \in \left[\lambda_{1},\theta_{m}\right]\setminus \big(\theta_k,\theta_{\widehat{k}+1}\big).
    \end{equation}
    Let $V_{\ell}$ be an orthonormal basis of $\spa{s_{1},\dotsc,s_{\widehat{k}}}+\mathcal{Q}(T_{m},e_{m},\Xi_{\widehat{k}})$. Then
    \begin{equation*}
        \norm{(I-PP^{\Ttran})W_{k}} < 2 \cdot \tol_{\mathrm{ra}} \quad {\mathrm{with}} \quad P=\diagm(V_{\ell},I).
    \end{equation*}
\end{theorem}
\begin{proof} 
    It suffices to show that $\norm{(I-PP^{\Ttran})w}\leq 2 \cdot \tol_{\mathrm{ra}}$ for every $w = W_{k} \alpha$ with $\alpha = [\alpha_1,\ldots,\alpha_k]^\Ttran$ such that $\norm{\alpha}=1$.
    By Cauchy interlacing, $\lambda_{1}\leq \lambda_{i}\leq \theta_{k} < \tau$ for $i = 1,\ldots, k$.
    Using the error bound~\cref{con:ratapp}, we have that 
    \begin{equation*}
        \norm{r_{\tau}(T_{n})w-w}^{2} =
        \Bignorm{\sum_{i=1}^{k}\alpha_{i}\bigl(r_{\tau}(\lambda_i)-1\bigr) w_{i}}^{2} = \sum_{i=1}^{k} \alpha_{i}^{2}\bigl(r_{\tau}(\lambda_{i})-\chi_{\tau}(\lambda_{i})\bigr)^{2} <  \tol_{\mathrm{ra}}^{2}. 
    \end{equation*} 
    This implies $\norm{(I-PP^{\Ttran})w}< \norm{(I-PP^{\Ttran})r_{\tau}(T_{n})w}+\tol_{\mathrm{ra}}$. It remains to show that the first term of the sum is also bounded 
    by $\tol_{\mathrm{ra}}$. For this purpose, we let 
    $\overline{w}\in\R^{m}$ contain the first $m$ components of $w$ and 
    recall from~\cref{thm:BKCS} that
    \begin{equation*}
        r_{\tau}(T_{n})w - \begin{bmatrix}
            r_{\tau}(T_{m})\overline{w}\\0
        \end{bmatrix}\in 
            \range\begin{bmatrix}
            V_{\ell}& 0 \\ 0 & I
        \end{bmatrix}=\range(P).
    \end{equation*}
    In turn,
    \begin{equation*}
        \begin{aligned}
            \norm{(I-PP^{\Ttran})r_{\tau}(T_{n})w} &= \bignorm{(I-V_{\ell}V_{\ell}^{\Ttran})r_{\tau}(T_{m})\overline{w}} =\Bignorm{(I-V_{\ell}V_{\ell}^{\Ttran})\sum_{i=1}^{m}r_{\tau}(\theta_{i})(s_{i}^{\Ttran}\overline{w})s_{i}} \\ 
            &= \Bignorm{(I-V_{\ell}V_{\ell}^{\Ttran})\sum_{i=\widehat{k}+1}^{m}r_{\tau}(\theta_{i})(s_{i}^{\Ttran}\overline{w})s_{i}}
            \leq \Bignorm{\sum_{i=\widehat{k}+1}^{m}r_{\tau}(\theta_{i})(s_{i}^{\Ttran}\overline{w})s_{i}},
        \end{aligned}
    \end{equation*}
    where we inserted the spectral decomposition $T_{m}=\sum_{i=1}^{m}\theta_{i}s_{i}s_{i}^{\Ttran}$ in the second equality and used $s_{i}\in\range(V_{\ell})$ for $1\leq i\leq \widehat{k}$ in the third equality. The proof is completed by  
    \begin{equation*}
        \Bignorm{\sum_{i=\widehat{k}+1}^{m}r_{\tau}(\theta_{i})(s_{i}^{\Ttran}\overline{w})s_{i}}^{2}
        =\sum_{i=\widehat{k}+1}^{m}r_{\tau}^{2}(\theta_{i})(s_{i}^{\Ttran}\overline{w})^{2} \leq \norm{\overline{w}}^{2}\max_{\widehat{k}+1\leq i\leq m}\abs{r_{\tau}(\theta_{i})}^{2}
        < \tol_{\mathrm{ra}}^{2},  
    \end{equation*}
    where we applied the error bound~\cref{con:ratapp} once more. 
\end{proof}
\end{paragraph}

\subsection{Reorthogonalization with fill-in}
\label{sec:reorth}

Roundoff error severely affects the Lanczos process and, in turn, the method presented in this work. In particular, the vectors 
produced by the three-term recurrence~\cref{eq:3recur} quickly lose numerical orthogonality. In the standard Lanczos method~\cite[p.~303]{Parlett1998}, this is fixed by \emph{immediately reorthogonalizing} each newly produced vector $q_{i+1}$ against \emph{all} previously generated vectors $q_1,\ldots, q_i$.
In principle, reorthogonalization destroys the three-term recurrence, since it involves more than only the last two vectors.
However, its immediate application ensures that the additional orthogonalization 
coefficients produced by reorthogonalization remain on the level of roundoff error~\cite[p.~350]{Stewart2001}.
This allows one to safely disregard these coefficients or, optionally, absorb them into the three-term recurrence. In turn, both the tridiagonal structure of $T_{\lan}$ and the numerical orthogonality of $Q_{\lan}$ in the Lanczos decomposition~\cref{eq:lanD} are preserved. 

In Lanczos with compression, we need to ensure that not only the orthogonality of the basis $Q_m$ but also the relation $Q_{m+1}^{\Ttran} F_m = 0$ in the Krylov-like decomposition~\eqref{eq:LCD} remain valid up to the level of roundoff error. While (repeated) reorthogonalization can be used for this purpose, there are two complications. Because reorthogonalization only has access to a compression of the Lanczos basis, one cannot expect that the three-term recurrence will be preserved (up to the level of roundoff error). Additionally, explicitly reorthogonalizing against $F_m$ would require us to store this matrix, which would roughly double the storage cost. 
In the following, we develop a reorthogonalization strategy that takes both issues into account. By allowing for fill-in in the
tridiagonal structure, we take care of the condition $Q_{m+1}^\Ttran F_m = 0$ without the need for storing $F_m$.%
\edit{\begin{remark}
    Besides full reorthogonalization, other strategies, such as selective reorthogonalization \cite{Parlett1979} and partial reorthogonalization \cite{Simon1984}, have also been explored for Lanczos method (with restarting). The following discussion focuses on full reorthogonalization because of its prevalence in state-of-the-art implementations, including 
    ARPACK \cite{lehoucq1998arpack} and the default Krylov-Schur solver in SLEPc~\cite{hernandez2007parallel}.
    \end{remark}
}

\begin{paragraph}{Reorthogonalization of Lanczos expansion}

In the following, we consider a single step of the Lanczos process that computes a new Lanczos vector $q_{i+2}$. Letting $c$ denote the round of the procedure~\cref{eq:compression} during which this vector is produced and assuming $c \ge 1$, we can write 
$i=c(m-\ell)+j$ with $\ell\leq j< m$. Suppose we start from a \emph{perturbed} Krylov-like decomposition
\begin{equation}
\label{eq:roe}
    A\widehat{Q}_{j} = \widehat{Q}_{j}\widehat{T}_{j}+\widehat{q}_{i+1}\widehat{b}_{j}^{\Ttran}+\widehat{F}_{j}.
\end{equation}
To increase readability, we have dropped the superscript $(c)$ for denoting the $c$th round.
We assume that 
$\widehat{T}_{j}$ is symmetric, $\widehat{b}_{j}=\widehat{Q}_{j}^{\Ttran}A\widehat{q}_{i+1}$, and 
\begin{equation} \label{eq:roe2}
    \widehat{Q}_{j+1}^{\Ttran}\widehat{Q}_{j+1}-I = \order(\mpr),\quad 
    \widehat{Q}_{j+1}^{\Ttran}\widehat{F}_{j} = \order(\norm{A}\mpr),\quad 
    \widehat{b}_{j} = \order(\norm{A}),
\end{equation}
with $\widehat{Q}_{j+1} = [\widehat{Q}_{j},\widehat{q}_{i+1}]$ and the unit roundoff $\mpr$.
Here and in the following, we use $\order(\cdot)$ to hide a constant that only depends on $n$ and grows mildly (at most polynomially) with $n$.  A decomposition~\eqref{eq:roe} satisfying these assumptions will be called a \emph{numerically stable Krylov-like decomposition}.
Our goal is to suitably modify the $(i+1)$th Lanczos step such that it preserves stability.
As we will see below, neither $\widehat{b}_{j}$ nor $\widehat{F}_{j}$ needs to be available explicitly for this purpose.

The Lanczos step that produces the $(i+2)$th vector proceeds by computing
\begin{equation*}
    \widecheck{w} = \fl(A\widehat{q}_{i+1}),\quad 
    \widecheck{\alpha}_{i+1} = \fl(\widehat{q}_{i+1}^{\Ttran}\widecheck{w}) ,\quad
    \widecheck{p} = \fl(\widecheck{w}-\widecheck{\alpha}_{i+1}\widehat{q}_{i+1}-\widehat{\beta}_{i}\widehat{q}_{i}), 
\end{equation*}
where $\fl(\cdot)$ denotes the result computed in floating-point arithmetic. Standard roundoff error analysis~\cite{Higham2002} implies that 
\begin{equation}
    \label{eq:rorp}
    \widecheck{p} = A\widehat{q}_{i+1}-\widecheck{\alpha}_{i+1}\widehat{q}_{i+1}-\widehat{\beta}_{i}\widehat{q}_{i} + \order(\norm{A}\mpr).
\end{equation}
However, $\widecheck{p}$ is \emph{not} ensured to be numerically orthogonal, neither to $\widehat{Q}_{j+1}$ nor to $\widehat{F}_{j}$. A common and effective way of ensuring numerical orthogonality to $\widehat{Q}_{j+1}$ is to perform full reorthogonalization as follows:
\begin{equation*}
    \widecheck{z} = \fl\bigl((I-\widehat{Q}_{j+1}\widehat{Q}_{j+1}^{\Ttran})\widecheck{p}\bigr),\quad 
    \widehat{\beta}_{i+1} = \fl(\norm{\widecheck{z}}),\quad 
    \widehat{q}_{i+2} = \fl(\widecheck{z}/\widehat{\beta}_{i+1}). 
\end{equation*}
Note that one step of reorthogonalization may not be enough; if $\widehat{q}_{i+2}$ is not numerically orthogonal to $\widehat{Q}_{j+1}$, one repeats this Gram--Schmidt process as described in \cite{Giraud2005}. After one or possibly several reorthogonalization steps, we obtain $\widehat{\beta}_{i+1}$ and $\widehat{q}_{i+2}$ satisfying
\begin{equation}
    \label{eq:roebq}
    \widehat{\beta}_{i+1}\widehat{q}_{i+2} = (I-\widehat{Q}_{j+1}\widehat{Q}_{j+1}^{\Ttran})\widecheck{p}+\order(\norm{A}\mpr),\quad
    \widehat{Q}_{j+2}^{\Ttran}\widehat{Q}_{j+2}-I=\order(\mpr)
\end{equation}
where $\widehat{Q}_{j+2}=[\widehat{Q}_{j+1},\widehat{q}_{i+2}]$. By \cref{eq:rorp},
we also have that $\widehat{\beta}_{i+1} = \order(\norm{A})$.
Combining \cref{eq:roebq} with the stable Krylov-like decompositions \cref{eq:roe}, we arrive at the decomposition 
\begin{equation}
    \label{eq:roex}
    A\widehat{Q}_{j+1} = \widehat{Q}_{j+1}\widecheck{T}_{j+1}+\widehat{\beta}_{i+1}\widehat{q}_{i+2}e_{j}^{\Ttran}+\widecheck{F}_{j+1},
\end{equation}
where
\begin{equation*}
    \widecheck{T}_{j+1} = \begin{bmatrix}
        \widehat{T}_{j} & \widehat{\beta}_{i}e_{j}\\ 
        \widehat{b}_{j}^{\Ttran} & \widecheck{\alpha}_{i+1}
    \end{bmatrix},\quad 
    \widecheck{F}_{j+1} \defi 
    A\widehat{Q}_{j+1} - \widehat{Q}_{j+1}\widecheck{T}_{j+1}-\widehat{\beta}_{i+1}\widehat{q}_{i+2}e_{j}^{\Ttran}
    = \begin{bmatrix}
        \widehat{F}_{j}& \order(\norm{A}\mpr)
    \end{bmatrix}.
\end{equation*}

The relation~\eqref{eq:roex} does \emph{not} constitute a numerically stable Krylov-like decomposition because $\widecheck{T}_{j+1}$ is not symmetric and 
the numerical orthogonality between $\widehat{Q}_{j+2}$ and $\widecheck{F}_{j+1}$ is not (yet) ensured.
In principle, both issues could be addressed by explicitly reorthogonalizing $\widehat{q}_{i+2}$ with respect to $\widehat{F}_{j}$, but this would require one to hold $\widehat{F}_{j}$. Instead, we update the (unknown) matrix $\widehat{F}_{j}$ implicitly. For this purpose, we replace $\widecheck{T}_{j+1}$ by 
\begin{equation*}
    \widehat{T}_{j+1} = \begin{bmatrix}
        \widehat{T}_{j} & \fl(\widehat{Q}_{j}^{\Ttran}A\widehat{q}_{i+1}) \\ 
        \fl(\widehat{Q}_{j}^{\Ttran}A\widehat{q}_{i+1})^{\Ttran} & \fl(\widehat{q}_{i+1}^{\Ttran}A\widehat{q}_{i+1})
    \end{bmatrix}.
\end{equation*}
Note that there is now fill-in in the last row and column, and $\widehat{T}_{j+1}$ is symmetric.
Let us define  
\begin{equation}  \label{eq:deffjp1}
    \widehat{F}_{j+1} = \begin{bmatrix}
        (I-\widehat{q}_{i+2}\widehat{q}_{i+2}^{\Ttran})\widehat{F}_{j} & 0
    \end{bmatrix}+\Delta,\quad 
    \widehat{b}_{j+1} =\widehat{Q}_{j+1}^{\Ttran}A\widehat{q}_{i+2},
\end{equation}
with
\begin{equation*}
    \Delta :=  A\widehat{Q}_{j+1} - \widehat{Q}_{j+1}\widehat{T}_{j+1}-\widehat{q}_{i+2}\widehat{b}_{j+1}^{\Ttran}- \begin{bmatrix}
        (I-\widehat{q}_{i+2}\widehat{q}_{i+2}^{\Ttran})\widehat{F}_{j} & 0
    \end{bmatrix}
\end{equation*}
Trivially, these definitions imply 
\begin{equation}
    \label{eq:roen}
    A\widehat{Q}_{j+1} = \widehat{Q}_{j+1}\widehat{T}_{j+1}+\widehat{q}_{i+2}\widehat{b}_{j+1}^{\Ttran}+\widehat{F}_{j+1}.
\end{equation}

If we can show that $\widehat{Q}_{j+2}^{\Ttran}\widehat{F}_{j+1}$ is small in norm then~\eqref{eq:roen} satisfies all the properties of a stable Krylov-like decomposition. To show this, we partition $\Delta=[\Delta_{j},\delta_{j+1}]$, such that $\delta_{j+1}$ is a vector.
Using $\fl(\widehat{Q}_{j}^{\Ttran}A\widehat{q}_{i+1}) = \widehat b_j+\order(\norm{A}\mpr)$, one computes
\begin{align}
 \Delta_{j} &= A\widehat{Q}_{j}-\widehat{Q}_{j}\widehat{T}_{j}-\widehat{q}_{i+1} \widehat b_j^\Ttran - \widehat{q}_{i+2}\widehat{q}_{i+2}^{\Ttran}A\widehat{Q}_{j}-(I-\widehat{q}_{i+2}\widehat{q}_{i+2}^{\Ttran})\widehat{F}_{j} + \order(\norm{A}\mpr) \nonumber \\
 &= - \widehat{q}_{i+2}\widehat{q}_{i+2}^{\Ttran}\big( A\widehat{Q}_{j}-\widehat{F}_{j} \big) + \order(\norm{A}\mpr) =
 - \widehat{q}_{i+2}\widehat{q}_{i+2}^{\Ttran}\big( \widehat{Q}_{j}\widehat{T}_{j}+\widehat{q}_{i+1}\widehat{b}_{j}^{\Ttran} \big) + \order(\norm{A}\mpr),  \label{eq:usekrylov} 
\end{align}
where~\eqref{eq:usekrylov} uses the stable Krylov-like decomposition~\cref{eq:roe} twice.
Because of the numerical orthogonality~\cref{eq:roebq} of $\widehat{Q}_{j+2}$, this establishes $\Delta_{j} = \order(\norm{A}\mpr)$.
For $\delta_{j+1}$, one computes
\begin{equation*}
    \begin{aligned}
        \delta_{j+1} &= A\widehat{q}_{i+1}-\widehat{Q}_{j}\fl(\widehat{Q}_{j}^{\Ttran}A\widehat{q}_{i+1})-\widehat{q}_{i+1}\fl(\widehat{q}_{i+1}^{\Ttran}A\widehat{q}_{i+1})-\widehat{q}_{i+2}\widehat{q}_{i+2}^{\Ttran}A\widehat{q}_{i+1}\\ 
        &= (I-\widehat{Q}_{j+2}\widehat{Q}_{j+2}^{\Ttran})A\widehat{q}_{i+1}+\order(\norm{A}\mpr) = \order(\norm{A}\mpr),
    \end{aligned}
\end{equation*}
where the last equality uses~\cref{eq:rorp,eq:roebq}.
Thus, we have shown that $\Delta=\order(\norm{A}\mpr)$ and, in turn, the definition~\cref{eq:deffjp1} together with~\cref{eq:roe2} yield
\begin{equation*}
    \widehat{Q}_{j+2}^{\Ttran}\widehat{F}_{j+1}
    =\begin{bmatrix}
        \widehat{Q}_{j+1}^{\Ttran}\widehat{F}_{j}&0\\ 
        0 &0
    \end{bmatrix}+\order(\norm{A}\mpr) = \order(\norm{A}\mpr).
\end{equation*} 

In summary, we have shown that~\cref{eq:roen} is indeed a stable Krylov-like decomposition. Thus, Lanczos expansion steps together with our reorthogonalization strategy maintain stable Krylov-like decompositions at the expense of admitting additional fill-in in the matrix $\widehat{T}_{j+1}$. This also applies to the classical Lanczos method carried out during the first round of the procedure~\eqref{eq:compression}, provided that the usual full reorthorgonalization is used.
\end{paragraph}

\begin{paragraph}{Compression step}
    In order to conclude that the entire procedure~\cref{eq:compression} results in stable Krylov-like decompositions, it remains to show that they are preserved by compression steps as well. For this purpose, consider a stable Krylov-like decomposition \emph{before}
    compression is carried out:
\[
    A\widehat{Q}_{m} = \widehat{Q}_{m}\widehat{T}_{m}+\widehat{q}_{i+1}\widehat{b}_{m}^{\Ttran}+\widehat{F}_{m},
\]
that is, the coefficients satisfy the properties stated above for~\cref{eq:roe}, with $j$ replaced by $m$.
Let $\widehat{V}_{\ell}$ denote the compression matrix computed by~\cref{alg:second_pole_strategy}. We assume that the block Gram--Schmidt process with reorthogonalization~\cite{Barlow2013} is used, ensuring that $\widehat{V}_{\ell}^{\Ttran}\widehat{V}_{\ell}-I=\order(\mpr)$.
Following~\cref{thmLan}, the compression step yields, when carried out in floating point arithmetic, a decomposition of the form
    \begin{equation}
        \label{eq:roec}
        A\widehat{\widetilde{Q}}_{\ell} = \widehat{\widetilde{Q}}_{\ell}\widehat{\widetilde{T}}_{\ell}+\widehat{q}_{i+1}\widehat{\widetilde{b}}_{\ell}^{\Ttran}+\widehat{\widetilde{F}}_{\ell},
    \end{equation}
    where $\widehat{\widetilde{b}}_{\ell} = \widehat{V}_{\ell}^{\Ttran}\widehat{b}_{m}$, 
    \begin{equation*}
        \widehat{\widetilde{Q}}_{\ell} = \fl(\widehat{Q}_{m}\widehat{V}_{\ell}),\quad 
        \widehat{\widetilde{T}}_{\ell} = \fl(\widehat{V}_{\ell}^{\Ttran}\widehat{T}_{m}\widehat{V}_{\ell}), \quad 
        \widehat{\widetilde{F}}_{\ell} =  \widehat{Q}_{m}(I-\widehat{V}_{\ell}\widehat{V}_{\ell}^{\Ttran})\widehat{T}_{m}\widehat{V}_{\ell}+\widehat{F}_{m}\widehat{V}_{\ell}+\widetilde{\Delta}
    \end{equation*}
    with $\widetilde{\Delta}=\order(\norm{A}\mpr)$. 
    From a basic roundoff error analysis, it follows that~\cref{eq:roec} is again a stable Krylov-like decomposition. 
\end{paragraph}

\begin{paragraph}{Backward stability}
It is well known that full reorthogonalization ensures the backward stability of the classical Lanczos method, the Arnoldi method, and the Krylov--Schur method; see, e.g.,~\cite[Thm.~4.1]{stewart2002krylov}~and~\cite[Chap.~5, Thm.2.5]{Stewart2001}.
As shown above, our reorthogonalization strategy ensures that Lanczos with compression returns a stable Krylov-like decomposition~\cref{eq:roe}. 
The following theorem shows that this implies backward stability, that is, Lanczos with compression returns an exact Krylov-like decomposition for slightly perturbed matrices.
\begin{theorem}
    \label{thmBS}
    Given a stable Krylov-like decomposition~\cref{eq:roe}, there is a matrix $[Q_j,q_{i+1}]$ with orthonormal columns
    and a matrix $E$ such that
    \[
     \big[\widehat{Q}_{j}, \widehat q_{i+1}\big] - [Q_j,q_{i+1}] = \order(\mpr),\quad E = \order(\norm{A}\mpr),
    \]
and
\begin{equation}
\label{eq:bs}
    (A+E){Q}_{j} = {Q}_{j}\widehat{T}_{j}+{q}_{i+1}\widehat{b}_{j}^{\Ttran}+{F}_{j},
\end{equation}
for some ${F}_{j}$ with $[Q_{j},q_{i+1}]^{\Ttran} F_j = 0$ and $\widehat F_j - F_j = \order(\norm{A}\mpr)$.
\end{theorem}
\begin{proof}

    Letting $\widehat{Q}_{j+1} = [\widehat{Q}_{j},\widehat{q}_{i+1}]$, we define
    \begin{equation*}
        Q_{j+1} := [Q_{j},q_{j+1}] \defi \widehat{Q}_{j+1}(\widehat{Q}_{j+1}^{\Ttran}\widehat{Q}_{j+1})^{-1/2}.
    \end{equation*}
    This matrix clearly has orthonormal columns and satisfies
    \begin{equation*}
        Q_{j+1}-\widehat{Q}_{j+1}
         = \widehat{Q}_{j+1}\bigl((\widehat{Q}_{j+1}^{\Ttran}\widehat{Q}_{j+1})^{-1/2}-I\bigr) = \order(\mpr), 
    \end{equation*}
    using the numerical orthogonality~\eqref{eq:roe2} of $\widehat{Q}_{j+1}$. Setting
    $
        F_{j} \defi (I-Q_{j+1}Q_{j+1}^{\Ttran})\widehat{F}_{j}
    $
    ensures $Q_{j+1}^{\Ttran}F_{j}=0$ and 
    \begin{equation*}
        \widehat{F}_{j}-F_{j} =  Q_{j+1}Q_{j+1}^{\Ttran}\widehat{F}_{j} = Q_{j+1}(\widehat{Q}_{j+1}^{\Ttran}\widehat{Q}_{j+1})^{-1/2}\widehat{Q}_{j+1}^{\Ttran}\widehat{F}_{j} = \order(\norm{A}\mpr),   
    \end{equation*}
    using $\widehat{Q}_{j+1}^{\Ttran}\widehat{F}_{j}=\order(\norm{A}\mpr)$ from~\eqref{eq:roe2}. Finally, setting
    \begin{equation*}
        E \defi  (Q_{j}\widehat{T}_{j}+q_{i+1}\widehat{b}_{j}^{\Ttran}+F_{j}-AQ_{j})Q_{j}^{\Ttran},
    \end{equation*}
    ensures that~\cref{eq:bs} is valid and
        \begin{equation*}
        E = \bigl((Q_{j}-\widehat{Q}_{j})\widehat{T}_{j}+(q_{i+1}-\widehat{q}_{i+1})\widehat{b}_{j}^{\Ttran}+(F_{j}-\widehat{F}_{j})-A(Q_{j}-\widehat{Q}_{j})\bigr)Q_{j}^{\Ttran} = \order(\norm{A}\mpr).
    \end{equation*}
    Here, we used $\widehat{T}_{j}=\order(\norm{A})$, which is implied by multiplying the stable Krylov-like decomposition \cref{eq:roe} from the left with $\widehat{Q}_{j}^{\Ttran}$.
\end{proof}

\cref{thmBS} has the same implications as existing backward stability results for the Krylov--Schur method~\cite[Chap.~5, Thm.2.5]{Stewart2001}. In particular, it implies
\begin{equation*}
    \widehat{T}_{j}-Q_{j}^{\Ttran}AQ_{j} = Q_{j}^{\Ttran}EQ_{j} = \order(\norm{A}\mpr),
\end{equation*}
showing that the computed Ritz values (computed with a backward stable eigensolver applied to $\widehat{T}_{j}$)
are Ritz values of a slightly perturbed matrix $A$. 
\emph{Assuming}, for simplicity, that $\range(Q_{j})$ contains an exact eigenvector of $A$, the Rayleigh--Ritz procedure can extract an eigenvalue/eigenvector approximation up to an error of $\order(\norm{A}\mpr)$.

\end{paragraph}
\subsection{Pseudocode}
\label{sec:resapp}
Combining the compression mechanism from \cref{sec:restart}, the adaptive strategy for choosing $V_{\ell}$ from \cref{sec:Vell}, and the reorthogonalization with fill-in from \cref{sec:reorth}, we arrive at~\cref{alg:LC}, the main algorithm of this work.

\begin{algorithm2e}[H]
    \caption{Lanczos with compression}
    \label{alg:LC}
    \KwIn{Symmetric matrix $A$. Tolerance for rational approximation $\tol_{\mathrm{ra}}$. Tolerance for residual $\tol_{\mathrm{res}}$. Integer $k$ for the number of desired eigenvalues. Integer $m$ for the maximum allowed dimension before each compression.}
    \KwOut{Ritz values/vectors approximating the desired eigenvalues/eigenvectors.}
    Choose $q_{1}$ as normalized Gaussian random vector\;
    Set $\boldsymbol\alpha_{0}=\boldsymbol\beta_{0}=[]$, $Q_{1}^{(0)}=[q_{1}]$, $\beta_{0}=0$, $c=0$ and $j=1$\;

        \For{$i=1,2,3,\dotsc,$}{
            Perform Lanczos  step to compute relation $\widecheck{q}_{i+1} = \beta_{i}^{-1}(Aq_{i}-\alpha_{i}q_{i}-\beta_{i-1}q_{i-1})$\; 
            Perform reorthogonalization with fill-in, described in \cref{sec:reorth}, to obtain $q_{i+1}$ and $T_{j}^{(c)}$\;
            Update $\boldsymbol\alpha_{i} = [\boldsymbol\alpha_{i-1},\alpha_{i}]$ and $\boldsymbol\beta_{i} = [\boldsymbol\beta_{i-1},\beta_{i}]$\;
            \If{$j=m$}
            {
                Check convergence; see~\cref{sec:checkconvergence} below\; 
                Apply \cref{alg:second_pole_strategy} to $T_{m}^{(c)}$ to compute orthonormal compression matrix $V_{\ell}^{(c)}$\;
                Compute $Q_{\ell}^{(c+1)} = Q_{m}^{(c)}V_{\ell}^{(c)}$ and $T_{\ell}^{(c+1)} = (V_{\ell}^{(c)})^{\Ttran}T_{m}^{(c)}V_{\ell}^{(c)}$\;
                Set $c=c+1$ and $j=\ell$\;
            }
            Update $Q_{j+1}^{(c)}=[Q_{j}^{(c)},q_{i+1}]$ and $j=j+1$\;
            
        }
    \Return{Ritz values/vectors associated with the smallest $k$ eigenvalues of $T_{m}^{(c)}$\;}
\end{algorithm2e}
\begin{remark}
    Several remarks about \cref{alg:LC} are in order:
    \begin{itemize}
        \item Both lines~4 and 5 require the product of $A$ with the last vector $q_{i}$ in the Lanczos basis. Thus, one matrix-vector multiplication with $A$ is required in each loop. 
        \item Line~6 collects the coefficients $\alpha_{i}$ and $\beta_{i}$ from the Lanczos process computed in line~4. Instead, one could also use the $(j,j-1)$ and $(j,j)$ entries of $T_{j}^{(c)}$, that is, the corresponding coefficients \emph{after} reorthogonalization with fill-in, but this only makes a minor difference; see also~\cite[p.~350]{Stewart2001}.
        \item In \cref{eq:Tcompression}, we assumed that $(V_{\ell}^{(0)})^{\Ttran} e_{m} = e_{\ell}$ in order to preserve sparsity in the $(\ell+1)$th column and row of $T_{m}^{(1)}$. However, when reorthogonalization with fill-in is performed afterwards, this sparsity will be destroyed anyway. Therefore, in line~9, we can choose $V_{\ell}^{(c)}$ as an arbitrary orthonormal basis of its range, without necessarily enforcing $(V_{\ell}^{(c)})^{\Ttran} e_{m} = e_{\ell}$
        \item The number of vectors $\ell$ retained after compression in line 9 may vary in the course of the algorithm, as it depends on the output $V_{\ell}$ of \cref{alg:second_pole_strategy}.
        \item The memory required by the algorithm is dominated by the need for storing $Q_{m}^{(c)}$ and $q_{i+1}$, which contain $(m+1)n$ entries in total. As explained
        in~\cref{sec:reorth}, there is no need to store the coefficients $b_{m}$ and $F_{m}$ of the Krylov-like decomposition~\cref{eq:LCD}.
        \item After exiting from line~8, we compute the (partial) spectral decomposition $T_{m}^{(c)}S_{k}^{(c)}=S_{k}^{(c)}\Theta_{k}^{(c)}$, where $\Theta_{k}^{(c)}$ contains the smallest $k$ eigenvalues of $T_{m}^{(c)}$.  Then we obtain the Ritz values from the diagonal entries of $\Theta_{k}^{(c)}$ and the Ritz vectors from the columns of $Q_{m}^{(c)}S_{k}^{(c)}$.
    \end{itemize}
\end{remark}

\subsubsection{Checking convergence} \label{sec:checkconvergence}

Line~8 of~\cref{alg:LC} checks the convergence of the computed Ritz values/vectors and exits the loop once $k$ converged Ritz vectors have been found.
In the following, we develop a suitable convergence criterion, based on an existing approach for Lanczos.

\paragraph{Checking convergence in standard Lanczos}

In the standard Lanczos method, it is common practice to check the residual for declaring convergence. Letting $W_{\lan,k}\in\R^{\lan\times k}$ contain orthonormal eigenvectors belonging to the smallest $k$ eigenvalues of $T_{\lan}$, the matrix $U_{\lan}=Q_{\lan}W_{\lan,k}\in\R^{n\times k}$ contains the corresponding Ritz vectors and the diagonal matrix
$U_{\lan}^{\Ttran}AU_{\lan}$ contains the corresponding Ritz values.
Thanks to the Lanczos decomposition~\cref{eq:lanD}, the residual norm for these Ritz value/vector pairs satisfies
\begin{equation} \label{eq:reslan}
    \norm{AU_{\lan}-U_{\lan}U_{\lan}^{\Ttran}AU_{\lan}} = \norm{AQ_{\lan}W_{\lan,k}-Q_{\lan}T_{\lan}W_{\lan,k}} = \abs{\beta_{\lan}} \norm{e_{\lan}^{\Ttran}W_{\lan,k}}.
\end{equation} 
Thus, convergence can be checked inexpensively, by inspecting $\abs{\beta_{\lan}} \norm{e_{\lan}^{\Ttran}W_{\lan,k}}$.

\paragraph{Checking convergence in Lanczos with compression (line~8 of \cref{alg:LC})}

In principle, the expression~\cref{eq:reslan} easily extends 
to Lanczos with compression by replacing the Lanczos decomposition with the Krylov-like decomposition \cref{eq:LCD}. However, this would require us to explicitly store the matrix $F_{m}$, essentially doubling storage cost.
To avoid this, we use the coefficients $\boldsymbol\alpha_{i}$ and $\boldsymbol\beta_{i}$, collected in line~6 of \cref{alg:LC}, to form the tridiagonal matrix $T_{i}\in\R^{i\times i}$ that matches the matrix $T_{\lan}$ from~\cref{eq:lanD1} generated by the underlying Lanczos process. We then compute the eigenvectors $W_{i,k}$ corresponding to the smallest $k$ eigenvalues of $T_{i}$. In accordance with~\cref{eq:reslan}, we approximate the residual norm for Lanczos with compression by the quantity $\abs{\beta_{i}}\norm{e_{i}^{\Ttran}W_{i,k}}$ \edit{without evaluating the matrix $F_{m}$}. 
In particular, one exits the loop of \cref{alg:LC} once $\abs{\beta_{i}}\norm{e_{i}^{\Ttran}W_{i,k}}\leq \tol_{\mathrm{res}}$ holds.

Our numerical experiments consistently show (see \cref{figexperr} below) that the described residual approximation strategy is reliable and effective. 
In fact, the results in~\cref{sec:convergence} below show that, in terms of Ritz values, standard Lanczos and Lanczos with compression have a similar convergence behavior (in exact arithmetic) up to the level of $\tol_{\mathrm{ra}}^{2}$. This suggests that the convergence of these two methods occurs nearly simultaneously, particularly the residual norm for standard Lanczos method should be small when Lanczos with compression converges.

\paragraph{Choice of $\tol_{\mathrm{ra}}$}

The eigenvalue approximation error of a Ritz value satisfies a bound~\cite[Chap.~4, Thm.~2.17]{Stewart2001}  that is proportional to the squared residual norm, which is approximately bounded by $\tol_{\mathrm{res}}^{2}$.
On the other hand, \cref{thmcon} below shows that
the difference between the Ritz values of the Lanczos method and Lanczos with compression is bounded, up to a constant, by $\tol_{\mathrm{ra}}^{2}$. 
Thus, it is reasonable to choose $\tol_{\mathrm{ra}} \approx \tol_{\mathrm{res}}$. 
In practice, we recommend choosing $\tol_{\mathrm{ra}}$ slightly smaller than $\tol_{\mathrm{res}}$. For example, in our numerical experiments in \cref{sec:numericalresults}, we set  $\tol_{\mathrm{ra}} = 10^{-6}$ for computing one or four eigenvalues of discrete Laplacian operators, and $\tol_{\mathrm{ra}}=10^{-7}$ for computing several (up to hundreds of) eigenvalues for matrices from density function theory. Because $\ell$ depends at most logarithmically on $\tol_{\mathrm{ra}}^{-1}$, this does not significantly increase computational cost and helps avoiding potential imbalances due to the involved constants.

\section{Convergence analysis} \label{sec:convergence} 

This section studies the impact of compression on the convergence of the Ritz values returned by~\cref{alg:LC} as well as its consequences for the complexity. 
Suppose that \cref{alg:LC} performs $\lan$ matrix-vector multiplications in total. 
Assuming that $U$ contains the orthonormal basis of Ritz vectors returned by \cref{alg:LC}, the quantity $\Tr(U^{\Ttran}AU)-\sum_{i=1}^{k}\lambda_{i}$ measures the approximation error of the corresponding Ritz values. Recalling the Lanczos decomposition defined in \cref{eq:lanD}, this error can be decomposed as 
\begin{equation}
    \label{eq:decompErr}
    \Tr(U^{\Ttran}AU)-\sum_{i=1}^{k}\lambda_{i} = \Bigl(\Tr(U^{\Ttran}AU)-\sum_{i=1}^{k}\lambda_{i}(T_{\lan})\Bigr)
    +\sum_{i=1}^{k}\bigl(\lambda_{i}(T_{\lan})-\lambda_{i}\bigr),
\end{equation} 
where  $\lambda_{i}(T_{\lan})$ denotes the $i$th smallest eigenvalue of $T_{\lan}$.
Because of $\range (U)\subset \mathcal{K}_{\lan}(A,q_{1})$, the Cauchy interlacing theorem implies that both terms in the right-hand side are non-negative.
The second term captures the convergence of the standard Lanczos method, which is well understood for the case $k = 1$; see, e.g.,~\cite{Saad1980, Parlett1998}. The analysis for the case $k > 1$
is more involved and assumes that all the first $k+1$ eigenvalues of $A$ are simple~\cite{kressner2024randomized}.

It remains to analyze the first term in \cref{eq:decompErr}, which corresponds to the loss from compression. 
The following analysis of this term disregards the effect of roundoff error.
In related work \cite{casulli2025lanczos} that proposed and analyzed a compressed Lanczos method for solving Lyapunov matrix equations, Paige's analysis~\cite{Paige1980} was used to incorporate the effects of roundoff error. However, it is unclear how to perform an analogous analysis  for eigenvalue computation.

\subsection{Loss from compression}

Suppose that we perform $\compress$ compressions in total and let $V^{(c)}_{\ell}$ denote the orthonormal matrix used in the $c$th compression, and $P_{c}=\diagM{V^{(c)}_{\ell},I}$ for $c=0,\dotsc,\compress-1$. To quantify the effect of compressions on the Ritz values, we recursively define $T^{(c)}$ as
\begin{equation}
    \label{eq:defT^}
    T^{(0)}=T_{\lan} \quad\text{and}\quad  T^{(c+1)} = P_{c}^{\Ttran}T^{(c)}P_{c}\quad \text{for }c=0,\dotsc,\compress-1,
\end{equation}
and let $W_{k}^{(c)}$ contain the orthonormal eigenvectors corresponding to the smallest $k$ eigenvalues of $T^{(c)}$. Then \cref{thm1} shows that the impact of each compression step on $W_{k}^{(c)}$ is bounded by
\begin{equation}
    \label{eq:normWkj}
    \norm{W_{k}^{(c)}-P_{c}P_{c}^{\Ttran}W_{k}^{(c)}} \leq 2 \cdot \tol_{\mathrm{ra}}.
\end{equation}
The following auxiliary lemma will allow us to take into account that $P_{c}P_{c}^{\Ttran}W_{k}^{(c)}$ is \emph{not} orthonormal.

\begin{lemma}
    \label{lem:nonOrthX}
    Let $T\in\R^{n\times n}$ be a symmetric matrix with eigenvalues $\mu_{1}\leq \dotsb \leq \mu_{n}$, and let $W_{k}=[w_{1},\dotsc,w_{k}]$ contain orthonormal eigenvectors associated with $\mu_1,\ldots,\mu_k$. For any matrix $X \in \R^{n\times k}$ such that $\norm{X-W_{k}}\leq \epsilon<1$, we have that
    \begin{equation*}
        \Rq_{T}(X)-\sum_{i=1}^{k}\mu_{i}\leq k(\mu_{n}-\mu_{1})\epsilon^{2},
    \end{equation*}
    where $\Rq_{T}(X)$ is the (scalar) Rayleigh quotient~\cite[Eq. (4.2)]{absil2002grassmann}  defined as 
\begin{equation*}
    \Rq_{T}(X) = \Tr\bigl((X^{\Ttran}X)^{-1}(X^{\Ttran}TX)\bigr).
\end{equation*}
\end{lemma} 
\begin{proof}
    Let $W_{\perp}$ be an orthonormal basis of the orthogonal complement of the range of $W_{k}$. Then we know that
    \begin{equation}
    \label{eq:nonorthX1}
        \Bignorm{\begin{bmatrix}
            W_{k}^{\Ttran}X-I\\ 
            W_{\perp}^{\Ttran}X
        \end{bmatrix}}=
        \Bignorm{\begin{bmatrix}
            W_{k}&W_{\perp}
        \end{bmatrix}\begin{bmatrix}
            W_{k}^{\Ttran}X-I\\ 
            W_{\perp}^{\Ttran}X
        \end{bmatrix}}
        =\norm{X-W_{k}}\leq \epsilon,
    \end{equation}
    implying that $W_{k}^{\Ttran}X$ is nonsingular.
    Setting $E=W_{\perp}W_{\perp}^{\Ttran}X(W_{k}^{\Ttran}X)^{-1}$, the matrix $X$ admits the decomposition 
    \begin{equation*}
        X = W_{k}W_{k}^{\Ttran}X+W_{\perp}W_{\perp}^{\Ttran}X = (W_{k}+E)W_{k}^{\Ttran}X.
    \end{equation*}
    Note that $W_{k}^{\Ttran} E = W_{k}^{\Ttran} T E  = 0$.
Together with the homogeneity of Rayleigh quotient, see \cite[Prop.~4.1]{absil2002grassmann}, this implies
    \begin{equation*}
        \Rq_{T}(X) = \Rq_{T}(W_{k}+E) = \Tr\bigl((I+E^{\Ttran}E)^{-1}(D_{\mu}+E^{\Ttran}TE)\bigr),
    \end{equation*} 
    where $D_{\mu}=\diagM{\mu_{1},\dotsc,\mu_{k}}$.
    Then
    \begin{equation*}
        \begin{aligned}
            \Rq_{T}(X)-\sum_{i=1}^{k}\mu_{i} 
            &= \Tr\bigl((I+E^{\Ttran}E)^{-1}(D_{\mu}+E^{\Ttran}TE)-D_{\mu}\bigr)\\
            &= \Tr\bigl((I+E^{\Ttran}E)^{-1}(D_{\mu}+E^{\Ttran}TE-D_{\mu}-E^{\Ttran}ED_{\mu})\bigr)\\ 
            &= \Tr\Bigl((I+E^{\Ttran}E)^{-1}\bigl(E^{\Ttran}(T-\mu_{1}I)E-E^{\Ttran}E(D_{\mu}-\mu_{1}I)\bigr)\Bigr)\\ 
            &\leq \Tr\Bigl((I+E^{\Ttran}E)^{-1}\bigl(E^{\Ttran}(T-\mu_{1}I)E\bigr)\Bigr)\\
            &\leq k\norm{T-\mu_{1}I}\norm{E(I+E^{\Ttran}E)^{-1/2}}^{2}\\ 
            &= k(\mu_{n}-\mu_{1})\norm{(I+E^{\Ttran}E)^{-1}E^{\Ttran}E},
        \end{aligned}
    \end{equation*}
    where the first inequality uses the fact that the trace of $E^{\Ttran}E(D_{\mu}-\mu_{1}I)$, a product of two symmetric positive semi-definite matrices, is non-negative.
    By~\cref{eq:nonorthX1}, we have that $\norm{(W_{k}^{\Ttran}X-I)y}^{2}+\norm{W_{\perp}^{\Ttran}Xy}^{2}\leq \epsilon^{2}\norm{y}^{2}$ holds for any $y\in\R^{k}$. Thus,
    \begin{equation*}
        \norm{E}^{2} = \norm{W_{\perp}^{\Ttran}X(W_{k}^{\Ttran}X)^{-1}}^{2} =  \sup_{y\neq 0}\frac{\norm{W_{\perp}^{\Ttran}Xy}^{2}}{\norm{W_{k}^{\Ttran}Xy}^{2}} \leq \sup_{y\neq 0}\frac{\epsilon^{2}\norm{y}^{2}-\norm{(W_{k}^{\Ttran}X-I)y}^{2}}{\norm{W_{k}^{\Ttran}Xy}^{2}}.
    \end{equation*}
    By the Cauchy--Schwarz inequality,
    \begin{equation*}
        \epsilon^{2}\norm{y}^{2}-\norm{(W_{k}^{\Ttran}X-I)y}^{2} = (\epsilon^{2}-1)\norm{y}^{2}-\norm{W_{k}^{\Ttran}Xy}^{2}+2y^{\Ttran}W_{k}^{\Ttran}Xy\leq \frac{\epsilon^{2}}{1-\epsilon^{2}} \norm{W_{k}^{\Ttran}Xy}^{2},
    \end{equation*}
    implying that $\norm{E}^{2}\leq \epsilon^{2}/(1-\epsilon^{2})$.
    The lemma is proved by 
    \begin{equation*}
        \norm{(I+E^{\Ttran}E)^{-1}E^{\Ttran}E} = \norm{I-(I+E^{\Ttran}E)^{-1}}\leq 1-\frac{1}{\norm{E}^{2}+1} \leq \epsilon^{2}.
    \end{equation*}
\end{proof}

Using~\cref{lem:nonOrthX}, the following result establishes that the error on the Ritz values caused by compression on the Ritz values is controlled by the squared tolerance for rational approximation, up to some constant.
\begin{theorem}
    \label{thmcon}
    Given a symmetric matrix $A$ with eigenvalues in 
    the (gapped) increased ordering~\eqref{eq:ordereigenvalues}, and an initial vector $q_{1}$ with $\norm{q_{1}}=1$, let $T_{\lan}$ be the tridiagonal matrix in the Lanczos decomposition \cref{eq:lanD}.
    Let $U\in\R^{n\times k}$ contain the Ritz vectors return by Lanczos with compression (with the same initial vector $q_{1}$) after $\compress$ compression and $\lan$ matrix-vector multiplications. Then  
    \begin{equation*}
        \Tr(U^{\Ttran}AU)-\sum_{i=1}^{k}\lambda_{i}(T_{\lan})\leq 4k\compress(\lambda_{n}-\lambda_{1})\tol_{\mathrm{ra}}^{2},
    \end{equation*}
    where $0<\tol_{\mathrm{ra}}< 1/2$ is a uniform upper bound for the tolerance of rational approximation in \cref{con:ratapp}.
\end{theorem}
\begin{proof}
    Recall the definition of $P_{c}^{\Ttran}W_{k}^{(c)}$ in \cref{eq:defT^}, the condition $0< \tol_{\mathrm{ra}}< 1/2$ and \cref{eq:normWkj} yield that $P_{c}^{\Ttran}W_{k}^{(c)}$ is (column) full rank.
    By the Courant--Fischer minimax theorem~\cite[Thm.~8.1.2]{Golub2013}, 
    \begin{equation*}
        \sum_{i=1}^{k}\lambda_{i}(T^{(c+1)}) \leq \Rq_{T^{(c+1)}}(P_{c}^{\Ttran}W_{k}^{(c)}) = \Rq_{T^{(c)}}(P_{c}P_{c}^{\Ttran}W_{k}^{(c)}).
    \end{equation*}
    Using \cref{eq:normWkj} again and applying \cref{lem:nonOrthX} to $T=T^{(c)}$, $W_{k}=W_{k}^{(c)}$ and $X=P_{c}P_{c}^{\Ttran}W_{k}^{(c)}$ yields that 
    \begin{equation*}
        \begin{aligned}
            &\sum_{i=1}^{k}\bigl(\lambda_{i}(T^{(c+1)})-\lambda_{i}(T^{(c)})\bigr) \leq \Rq_{T^{(c)}}(P_{c}P_{c}^{\Ttran}W_{k}^{(c)})-\sum_{i=1}^{k}\lambda_{i}(T^{(c)})\\  
        \leq&\  4k\bigl(\lambda_{\max}(T^{(c)})-\lambda_{\min}(T^{(c)})\bigr)\tol_{\mathrm{ra}}^{2}.        
        \end{aligned}
    \end{equation*}
    Using Cauchy interlacing, we immediately obtain the loss of Ritz value after one compression as 
    \begin{equation*}
        \sum_{i=1}^{k}\bigl(\lambda_{i}(T^{(c+1)})-\lambda_{i}(T^{(c)})\bigr) \leq
        4k\bigl(\lambda_{\max}(T^{(c)})-\lambda_{\min}(T^{(c)})\bigr)\tol_{\mathrm{ra}}^{2}\leq 4k(\lambda_{n}-\lambda_{1})\tol_{\mathrm{ra}}^{2}.
    \end{equation*}
    Applying this inequality recursively and noting that 
    \begin{equation*}
        \Tr(U^{\Ttran}AU) = \sum_{i=1}^{k}\lambda_{i}(T^{(\compress)})\quad \text{and}\quad \lambda_{i}(T_{\lan}) = \lambda_{i}(T^{(0)}),
    \end{equation*}
    we finish the proof.
\end{proof}

\subsection{Complexity analysis of orthogonalization}

In this part, we present a back-of-the-envelope complexity analysis for the cost of (re)orthogonalization. Again, we ignore the effects of roundoff error. While roundoff may delay convergence, our numerical experiments indicate that this delay is often not significant.

For the Lanczos method without restarting or compression, we expect that up to $\iter_{\lan} \sim k / \sqrt{\gap_{k}}$ iterations are required to attain a fixed tolerance in the first $k$ eigenvalues. As mentioned above, this is a standard result for $k = 1$~\cite{Saad1980, Parlett1998}, while the corresponding result for $k > 1$ requires additional assumptions~\cite{kressner2024randomized}.
When full reorthogonalization is applied, the total number of operations required for orthogonalization scales as
\begin{equation*}
    \comp_{\lan} \sim n\cdot \iter^2_{\lan} \sim n k^2 / \gap_{k}.
\end{equation*}
In particular, a small relative gap may imply slow convergence and, consequently, higher orthogonalization costs.

According to the convergence result of \cref{thmcon}, Lanczos with compression requires essentially the same number of iterations as standard Lanczos, that is, $\iter_{\lan}\sim k / \sqrt{\gap_{k}}$, at least up to the level of the compression error.
However, in every iteration, we only need to orthogonalize with respect to the compressed basis, containing at most $m = \widehat{k} + \ell$ vectors. As discussed in \cref{sec:Vell}, the degree of rational approximation depends on the gap logarithmically, implying that $m\sim k+\log(1/\gap_{k})$.
Hence, the cost of orthogonalization reduces to 
\begin{equation*}
    \comp_{\lc} \sim n (m\cdot \iter_{\lan})  \sim n k\bigl(k + \log(1/\gap_{k})\bigr) / \sqrt{\gap_{k}}\sim nk^{2} / \sqrt{\gap_{k}}.
\end{equation*}
For small relative gaps, $0<\gap_{k}\ll 1$, we therefore expect a reduction of orthogonalization costs by up to a factor of roughly $1/\sqrt{\gap_{k}}$.

\section{Numerical results} \label{sec:numericalresults}
In this section, we compare~\cref{alg:LC}, Lanczos with compression (LC), with the Krylov--Schur method (KS). While Matlab's \texttt{eigs} (as of version R2017b) is based on KS, we found that the choice of the restart length in \texttt{eigs} can sometimes be suboptimal in terms of matrix-vector products (see discussion in~\cref{sec:kspara} below) and we therefore used our own implementation of KS\footnote{All numerical experiments in this section have been implemented in Matlab 2022b and were carried out on an AMD Ryzen~9 6900HX Processor (8 cores, 3.3--4.9 GHz) and 32 GB of RAM. Scripts to reproduce numerical results are publicly available at \url{https://github.com/nShao678/LanczosWithCompression-code}.}.
Both implementations, KS and LC, employ Gaussian random initial vectors and full reorthogonalization. The eigenvalues computed by Matlab's \texttt{eigs} (with default tolerance $10^{-14}$) are used as a reference when computing eigenvalue approximation errors.
The computational cost is measured in terms of the number of matrix-vector multiplications (referred to as matvecs). In turn, the improvement of LC over KS is quantified by the ratio
\begin{equation}
    \label{defimprovement}
    \text{Improvement}\defi 1-\frac{\text{$\#$matvecs required by LC}}{\text{$\#$matvecs required by KS}}.
\end{equation}
Since the impact of deflation on the number of matvecs is negligible, we perform both methods without deflation in all experiments.

\subsection{Experiments for Laplacian eigenvalue problem}

In this part, we consider a discretized 2D-Laplacian eigenvalue problem
\begin{equation*}
    -\Delta_{n}U_{k}=U_{k}\Lambda_{k},
\end{equation*}
where $\Delta_{n} \in \R^{n\times n}$ is a discretized Laplace operator on an $L$-shaped domain, generated by the Matlab command \texttt{A = 3/4*nx\^{}2*delsq(numgrid(`L',2+nx))}. We aim to compute an orthonormal basis $U_{k} \in \R^{n\times k}$ consisting of the eigenvectors corresponding to the smallest $k$ eigenvalues $\lambda_1\le \cdots \le \lambda_k$ of $-\Delta_{n}$. 
The maximum dimension of Krylov subspace before restarting is set to $m=60$ for both methods. 
The tolerance of rational approximation in LC is set to $\tol_{\mathrm{ra}}=10^{-6}$, except in \cref{sec:numexpCompressErr}.
Letting $\mu_1 \le \cdots \le \mu_k$ denote the $k$ smallest Ritz values obtained by LC or KS, we measure their relative error using
\begin{equation} \label{eq:errorritz}
\sum_{i=1}^{k}(\mu_{i}-\lambda_{i}) \,\bigg/\, \sum_{i=1}^{k}\lambda_{i}. 
\end{equation}

\subsubsection{Comparison with KS} \label{sec:kspara}

\paragraph{Restart length in KS}

In KS, one needs to fix a restart length $\ell$, that is, the number of Ritz vectors retained after restarting. In \texttt{eigs} (of Matlab 2022b), the choice of $\ell$ differs when computing one versus several eigenvalues. When computing a single eigenvalue ($k = 1$), $\ell$ is set to half the dimension of the Krylov subspace before restarting, that is, $\ell = m/2 = 30$ in our case. However, when computing $k > 1$ eigenvalues, $\ell$ is dynamically chosen between $k$ and $2k$, depending on the number of converged Ritz vectors.
To examine the influence of choosing different, fixed values for $\ell$, \cref{figexpLap,tabexpLap} present numerical results obtained when choosing
$n = 67\,500$ and varying $\ell$ between $k$ and $40$. Both KS and LC are terminated when 
the relative error~\eqref{eq:errorritz} is below a prescribed tolerance $\mathsf{tol}$. These results indicate that $\ell = m/2$ is a reasonable choice for both $k = 1$ and $k = 4$, at least in terms of the number of matvecs. 
In the remaining experiments, we therefore always choose $\ell=m/2$ for KS. Note that, even with this tuned choice for KS, 
LC continues to achieve a 4\%--7\% improvement relative to KS.

\begin{figure}[htbp]
    \centering
    \subfloat[Computing one eigenvalue]{
    \includegraphics[width=\figsizeD]{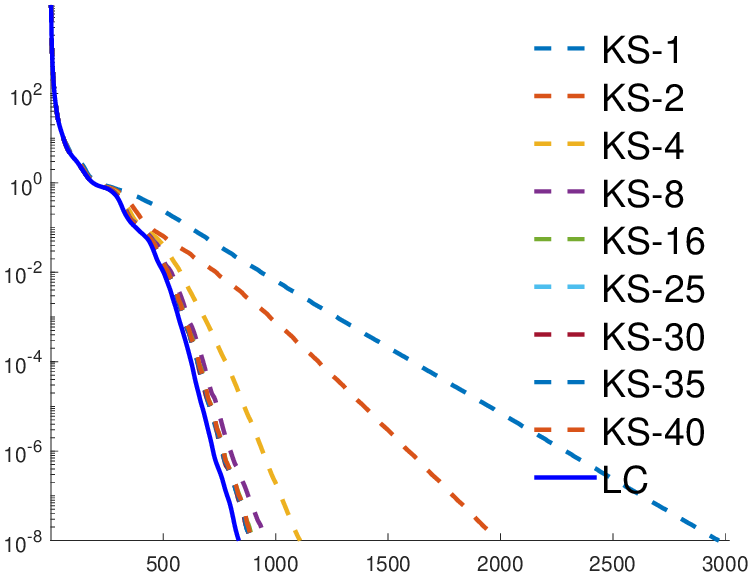}}
    \subfloat[Computing four eigenvalues]{
    \includegraphics[width=\figsizeD]{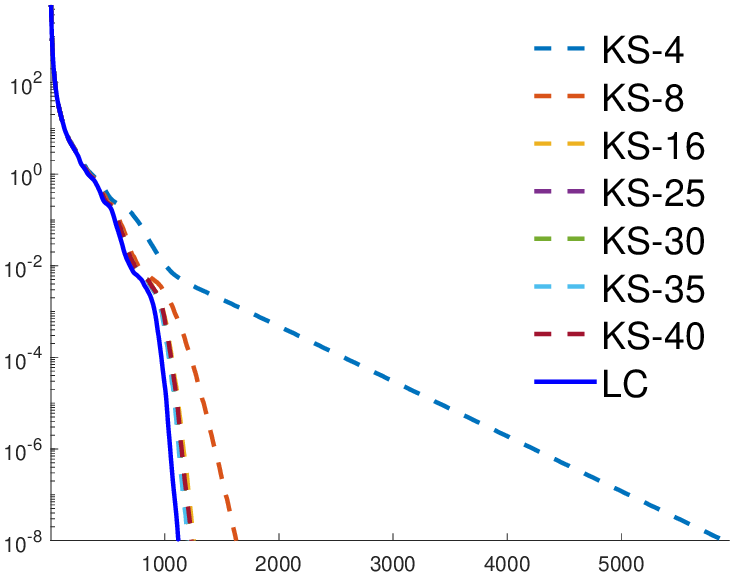}}
    \caption{Convergence history of Laplacian eigenvalue problems. The X-axis represents the number of matvecs, and the Y-axis represents the relative error~\cref{eq:errorritz} of the Ritz values for $k=1$ (left plot) and $k = 4$ (right plot). The maximum dimension of the Krylov subspace is $m = 60$. KS-$\ell$ denotes KS with $\ell$ Ritz vectors retained after restarting.}
    \label{figexpLap}
\end{figure}

\begin{table}[htbp]
    \centering
    \caption{$\#$matvecs required by KS and LC to ensure a relative eigenvalue approximation error~\cref{eq:errorritz} below $\mathsf{tol}$ for a Laplacian eigenvalue problem. The improvement~\cref{defimprovement} is computed for the minimal $\#$matvecs required by KS-$\ell$ (for all choice of $\ell$ in \cref{figexpLap}). The minimum value was always in 
    $\{25,30,35\}$ and is marked in bold.}
    \resizebox{\textwidth}{!}{
    \begin{tabular}{lcccccccccc}
        \toprule
          & \multicolumn{5}{c}{Computing $k = 1$ eigenvalue}  & \multicolumn{5}{c}{Computing $k = 4$ eigenvalues} \\
          \cmidrule(lr){2-6} \cmidrule(lr){7-11}
        $\tol$ & $10^{-4}$ & $10^{-5}$ & $10^{-6}$ & $10^{-7}$ & $10^{-8}$ & $10^{-4}$ & $10^{-5}$ & $10^{-6}$ & $10^{-7}$ & $10^{-8}$ \\
        \midrule
        KS-25  & 653  & \bf 707  & \bf 761  & \bf 830  & 890  & 1043 & 1098 & 1135 & 1176 & 1223 \\
        KS-30  & \bf 652  & \bf 707  & 764  & 831  & \bf 888  & 1041 & 1095 & 1132 & 1175 & 1219 \\
        KS-35  & 653  & 708  & 765  & 834  & 889  & \bf 1036 & \bf 1088 & \bf 1123 & \bf 1160 & \bf 1202 \\
        LC   & 625  & 673  & 722  & 785  & 837  & 971  & 1016 & 1048 & 1084 & 1119 \\
        \midrule
        Improvement & 4.1\% & 4.8\% & 5.1\% & 5.4\% & 5.7\% & 6.3\% & 6.6\% & 6.7\% & 6.6\% & 6.9\% \\
        \bottomrule
        \end{tabular}}
    \label{tabexpLap}
\end{table}

\paragraph{Behavior for increasing matrix size}

For $n$ ranging from $7\,500$ to $750\,000$, we report the number of required matvecs and the corresponding improvement, as defined in \cref{defimprovement}, for the relative errors of Ritz values falling below various tolerance levels in \cref{figlargen}.
This experiment shows that the performance of LC continues to improve as the matrix size $n$ increases.

\begin{figure}[htbp]
    \centering
    \subfloat[Computing one eigenvalue]{
    \includegraphics[width=\figsizeD]{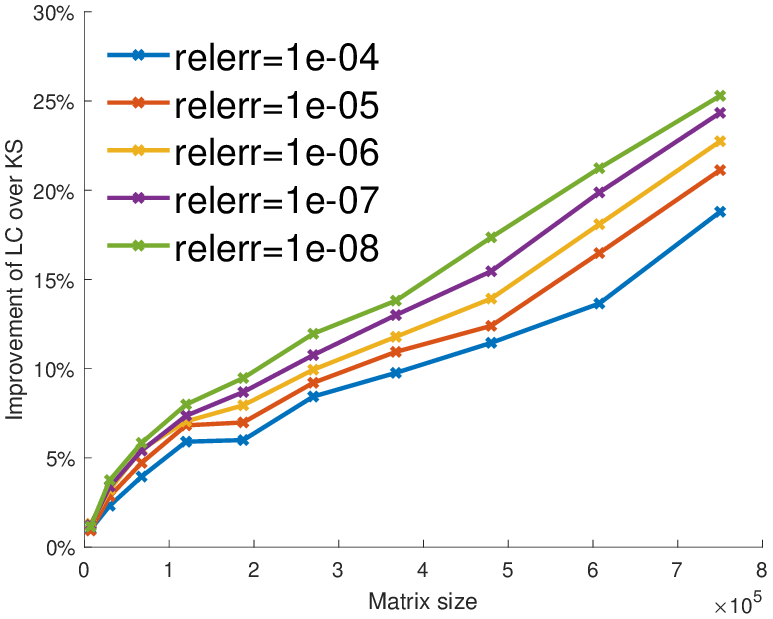}}
    \subfloat[Computing four eigenvalues]{
    \includegraphics[width=\figsizeD]{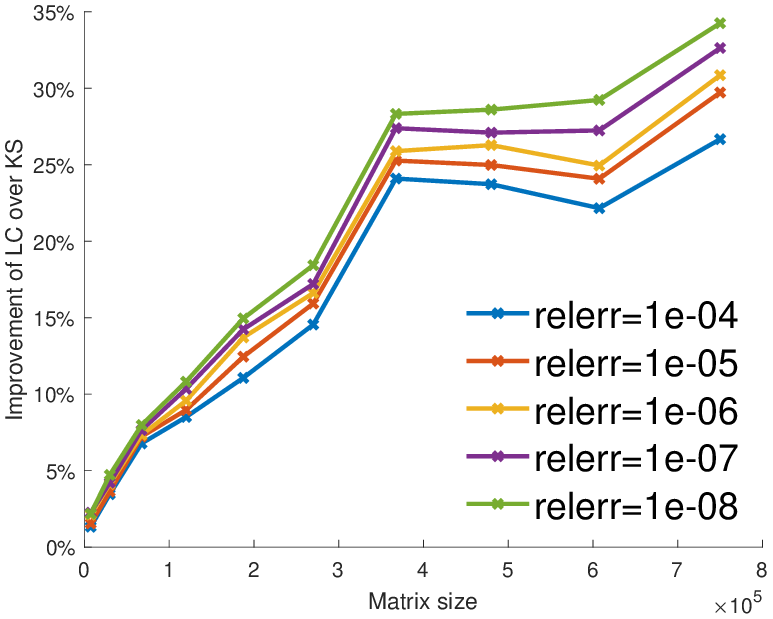}}
    \caption{Improvement of LC over KS, as defined in~\cref{defimprovement} as the size $n$ of discretized Laplacian eigenvalue  problem increases. The maximum dimension of the Krylov subspace is set to $60$.}
    \label{figlargen}
\end{figure}

\paragraph{Robustness with respect to random initial vectors}
At the end of this part, we show that the improvement is robust with respect to the choice of initials vectors. To verify that the results are not sensitive to 
the particularly chosen Gaussian random initial vector, we take $n=67\,500$ and $\ell=m/2=30$, and run KS and LC for $100$ trials of Gaussian random initial vectors. The resulting statistics are collected in \cref{tabStatGaussian}. 
All these experiments suggest that, for computing Laplacian eigenvalue problems on an $L$-shaped domain, the behavior of LC in terms of matvecs is never worse, always better, and sometimes significantly better than KS.
\begin{table}[htbp]
    \centering
    \caption{Improvement statistics for repeated trials with Gaussian random initial vectors. SD and CI stands for standard deviation and confidence intervals, respectively.}
    \label{tabStatGaussian}
    \resizebox{\textwidth}{!}{
    \subfloat[Computing $k = 1$ eigenvalue]{
    \begin{tabular}{cccc}
        \toprule 
        $\tol$ & mean & SD & $95\%$ CI \\ \midrule
        $10^{-4}$ & 3.87\% & 0.62\% & $[3.75\%, 3.99\%]$\\  
        $10^{-5}$ & 4.37\% & 0.54\% & $[4.26\%, 4.48\%]$ \\ 
        $10^{-6}$ & 4.88\% & 0.45\% & $[4.79\%, 4.97\%]$\\ 
        $10^{-7}$ & 5.31\% & 0.43\% & $[5.23\%, 5.40\%]$\\ 
        $10^{-8}$ & 5.63\% & 0.40\% & $[5.56\%, 5.71\%]$ \\ \bottomrule 
    \end{tabular}}
    \subfloat[Computing $k = 4$ eigenvalues]{
    \begin{tabular}{cccc}
        \toprule  
        $\tol$ & mean & SD & $95\%$ CI \\ \midrule
        $10^{-4}$ & 6.28\% & 0.43\% & $[6.19\%, 6.36\%]$ \\  
        $10^{-5}$ & 6.65\% & 0.42\% & $[6.57\%, 6.74\%]$ \\  
        $10^{-6}$ & 7.01\% & 0.39\% & $[6.93\%, 7.09\%]$\\  
        $10^{-7}$ & 7.31\% & 0.37\% & $[7.24\%, 7.38\%]$\\ 
        $10^{-8}$ & 7.62\% & 0.35\% & $[7.55\%, 7.69\%]$  \\ \bottomrule 
    \end{tabular}}}
\end{table}

\subsubsection{Approximation of residual}

In order to check convergence, we recall from \cref{sec:resapp} that the residual of LC is approximated by the residual of the standard Lanczos method (without restarting).
We validate the quality of this approximation by demonstrating that the behavior of these residuals is very similar in numerical examples. Taking $n = 67\,500$, we compute either the smallest eigenvalue or the smallest four eigenvalues.
In this experiment, we compute the residual (defined in \cref{eq:reslan}) for standard Lanczos and LC after each matvec, and report the results in \cref{figexperr}.
It is clear that in both cases LC closely follows the standard Lanczos method in terms of Ritz values and residuals. Thus, the approximation of the residual is reliable.

\begin{figure}[htbp]
    \centering

    \subfloat[$\mu_{1}-\lambda_{1}$ and $\sum_{i=1}^{4}(\mu_{i}-\lambda_{i})$]{
    \includegraphics[width=\figsizeF]{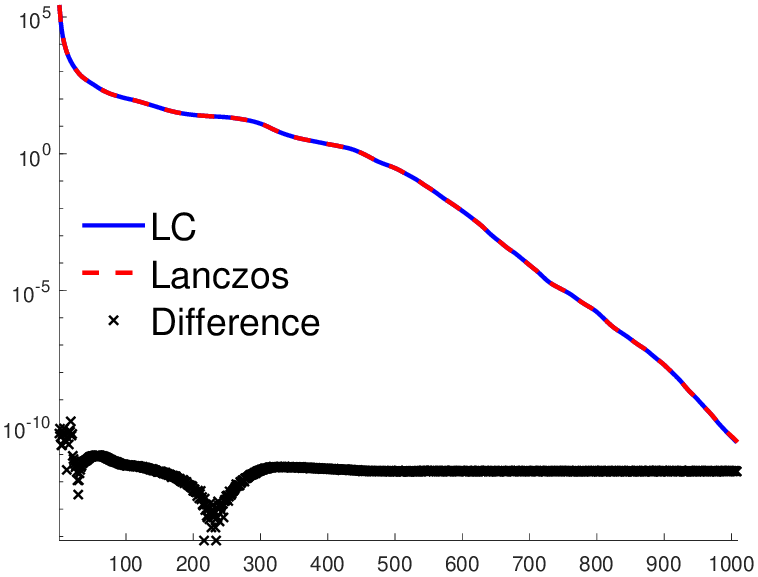}
    \includegraphics[width=\figsizeF]{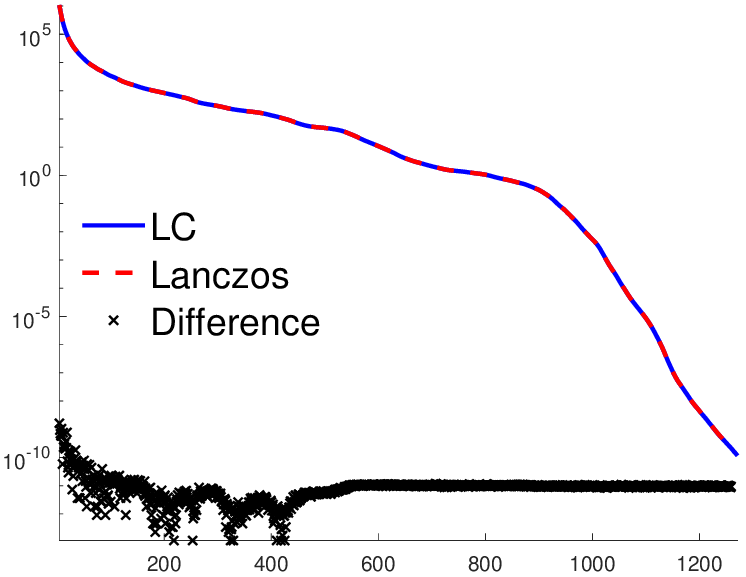}}
    \subfloat[Residual]{
    \includegraphics[width=\figsizeF]{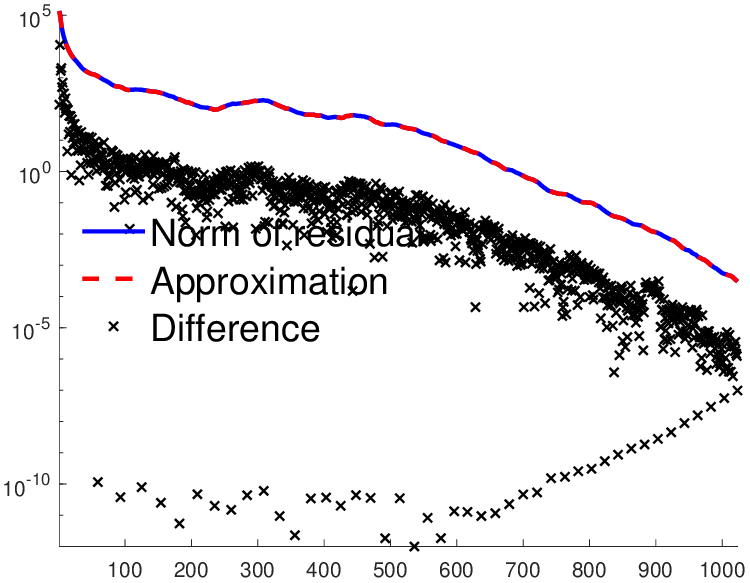}
    \includegraphics[width=\figsizeF]{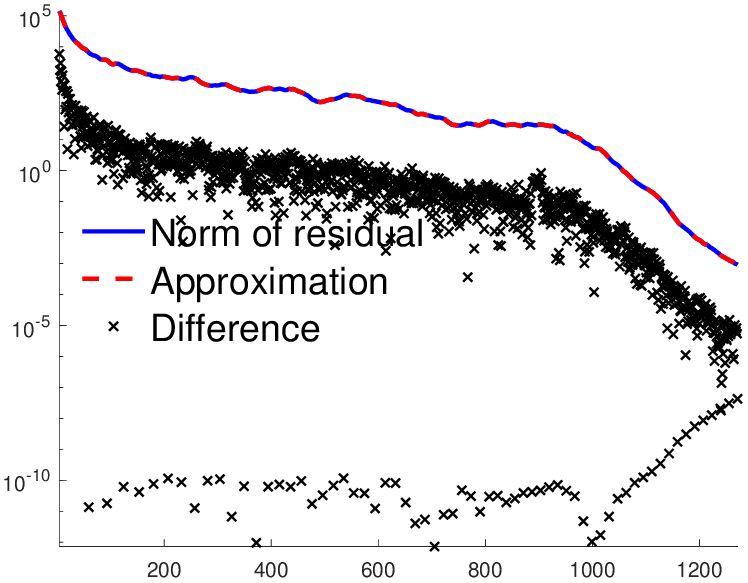}}
    \caption{Convergence history of Laplacian eigenvalue problems. 
    The X-axis represents the number of matvecs, while the Y-axis shows, for (a), the error of the Ritz values and, for (b), the norm of the residuals, for computing $k = 1$ eigenvalue (left) or $k = 4$ eigenvalues (right). The maximum dimension of the Krylov subspace for LC is set to $60$. The Lanczos method is implemented without restarting or reorthogonalization. The residual is computed after each matvec of LC.}
    \label{figexperr}
\end{figure}
\subsubsection{Influence of compression error}
\label{sec:numexpCompressErr}

\cref{thmcon} establishes that the accuracy of the Ritz values depends quadratically on the compression error $\tol_{\mathrm{ra}}$. We have verified this result through numerical experiments. Fixing $n=67\,500$ and varying $\tol_{\mathrm{ra}} = 10^{-i}$ for $i=1,\dotsc,5$, we record the convergence histories of the relative error in the Ritz values in \cref{fig:exptol}. 
The numerical results indicate that a large compression error leads to stagnation in the convergence of the Ritz values. This stagnation can be eliminated by employing a higher-quality rational compression (smaller $\tol_{\mathrm{ra}}$). Moreover, in this example, the observed dependence of the accuracy on the compression error is actually better than quadratic and appears to be nearly cubic.

\begin{figure}[htbp]
    \centering
    \subfloat[Computing $k = 1$ eigenvalue]{
    \includegraphics[width=\figsizeD]{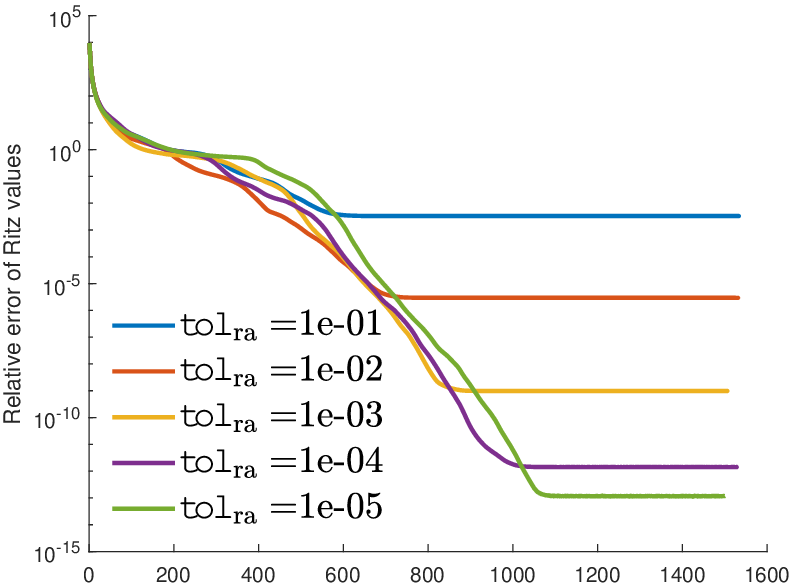}}
    \subfloat[Computing $k = 4$ eigenvalues]{
    \includegraphics[width=\figsizeD]{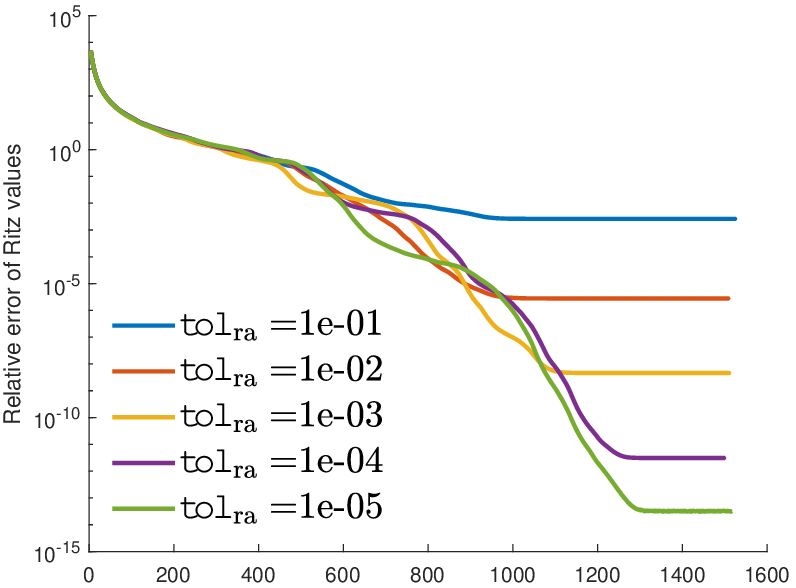}}
    \caption{Convergence histories of LC with different compression errors, when setting the maximum dimension of the Krylov subspace to $60$. The X-axis represents the number of matvecs.}
    \label{fig:exptol}
\end{figure}

\subsection{Test for matrices from density functional theory}
In this part, we select several matrices in \cite{Davis2011} originating from density functional theory to compare LC with KS.
In density functional theory, the number of eigenvalues of interest corresponds to the number of occupied states, as detailed for the test matrices in \cref{matnameDFT}.

\begin{table}[htbp]
    \centering
    \caption{Statistics on matrices from density functional theory}
    \label{matnameDFT}
    \begin{tabular}{crrr}
    \toprule
        name & size & nnz & \#occupied states\\ \midrule
        \texttt{H2O} & 67\,024&	2\,216\,736& 4	\\ 
        \texttt{GaAsH6} & 61\,349 &	3\,381\,809 &7	\\ 
        \texttt{SiO2} & 155\,331& 11\,283\,503& 8\\ 	 
        \texttt{Si5H12} & 19\,896 &	738\,598 & 16\\ 
        \texttt{Ga10As10H30} &  113\,081 & 6\,115\,633 & 55 \\
        \texttt{Si34H36}& 97\,569 & 5\,156\,379& 86\\
        \bottomrule
    \end{tabular}
\end{table}

Throughout this part, the tolerance for rational approximation is set to $\tol_{\mathrm{ra}} = 10^{-7}$. For the maximum dimension of the Krylov subspace in both methods, we set $m = \max\{80, 4k\}$, where $k$ is the number of occupied states. After restarting, we retain $\ell = m/2$ Ritz vectors for KS. Note that most matrices in \cref{matnameDFT} are indefinite. We examine the absolute error of the Ritz values $\sum_{i=1}^{k}(\mu_{i}-\lambda_{i})$, where $\mu_{i}$ and $\lambda_{i}$ are the $i$-th smallest Ritz value and eigenvalue of $A$, respectively.
(By Cauchy interlacing, each term is always nonnegative, that is, $\mu_{i}-\lambda_{i}\geq 0$.) The convergence histories are collected in \cref{figexpDFT}.

Although the same tolerance for rational approximation is applied to all test matrices, the accuracy of the Ritz values obtained from LC can vary significantly across different matrices, ranging from machine precision to $10^{-7}$. 
In particular, a stagnation happens when computing a large number of eigenvalues, such as $55$ for \texttt{Ga10As10H30} and $86$ for \texttt{Si34H36}.
We observed that this stagnation cannot be fixed by using a higher-quality compression.
It is important to note that stagnation is caused by the influence of roundoff error on convergence; this does not contradict the backward stability. 
This phenomenon indicates that the fill-in of $T_{m}$ destroys the rank-one structure of its off-diagonal block in \cref{thm:BKCS}, which is needed to achieve high accuracy.

\begin{figure}[htbp]
    \centering
    \subfloat[\texttt{H2O}(4)]{
    \includegraphics[width=\figsizeT]{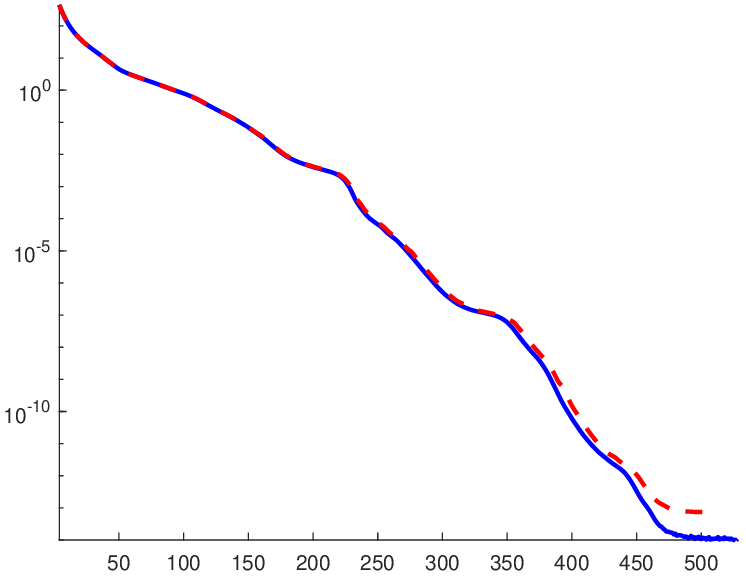}}
    \subfloat[\texttt{GaAsH6}(7)]{
    \includegraphics[width=\figsizeT]{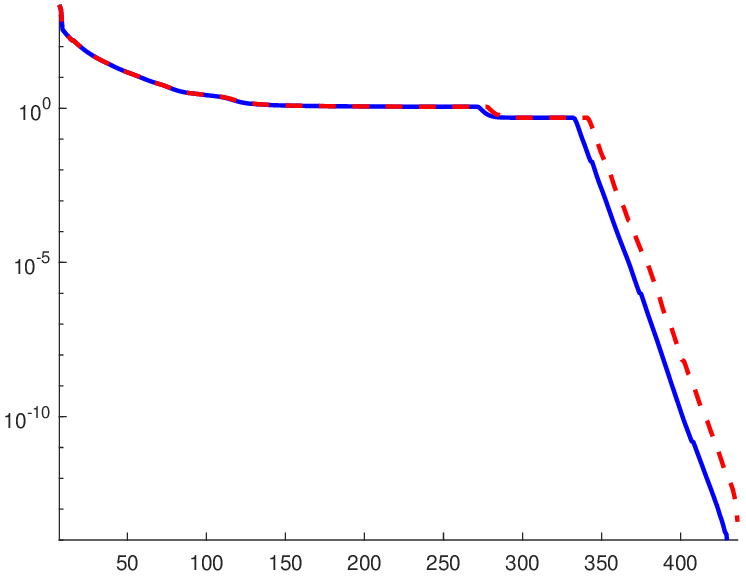}}
    \subfloat[\texttt{SiO2}(8)]{
    \includegraphics[width=\figsizeT]{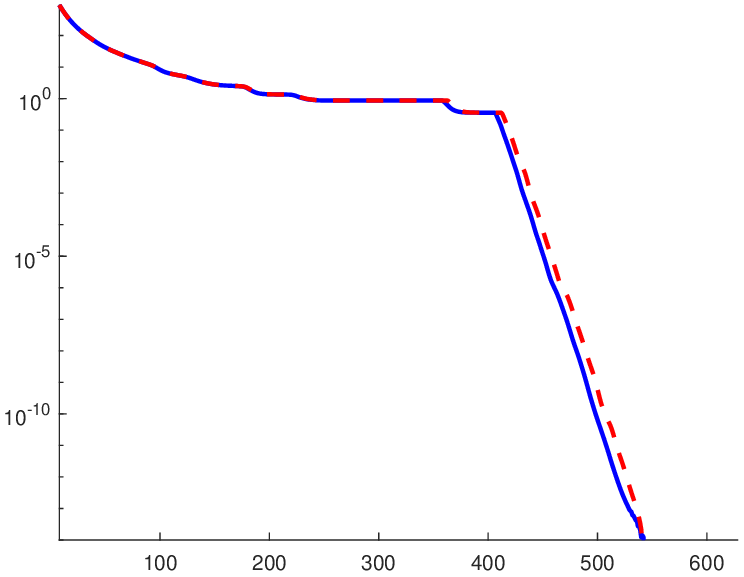}}

    \subfloat[\texttt{Si5H12}(16)]{
    \includegraphics[width=\figsizeT]{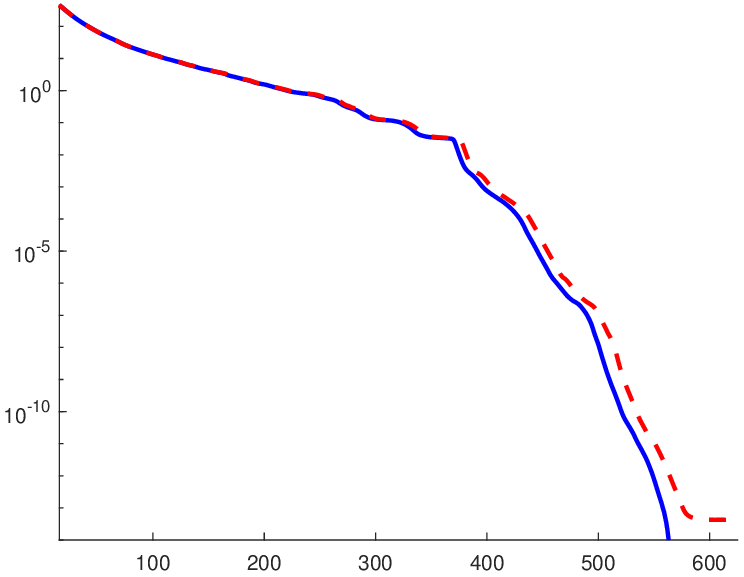}}
    \subfloat[\texttt{Ga10As10H30}(55)]{
    \includegraphics[width=\figsizeT]{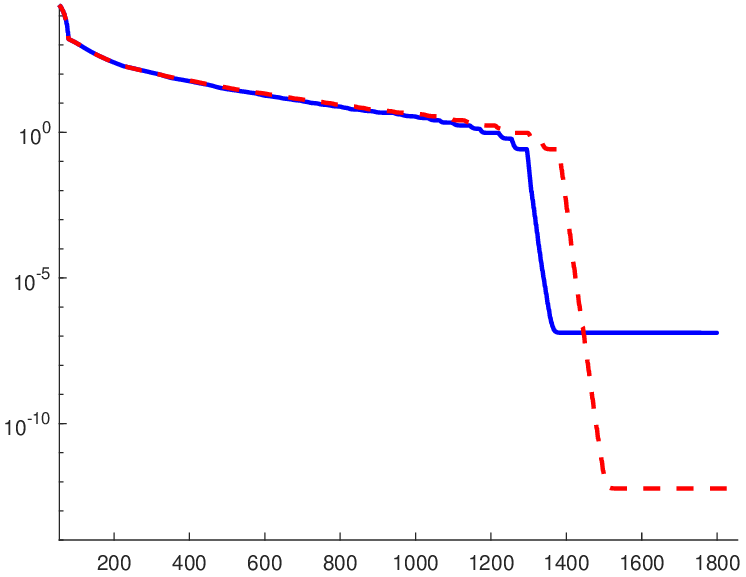}}
    \subfloat[\texttt{Si34H36}(86)]{
    \includegraphics[width=\figsizeT]{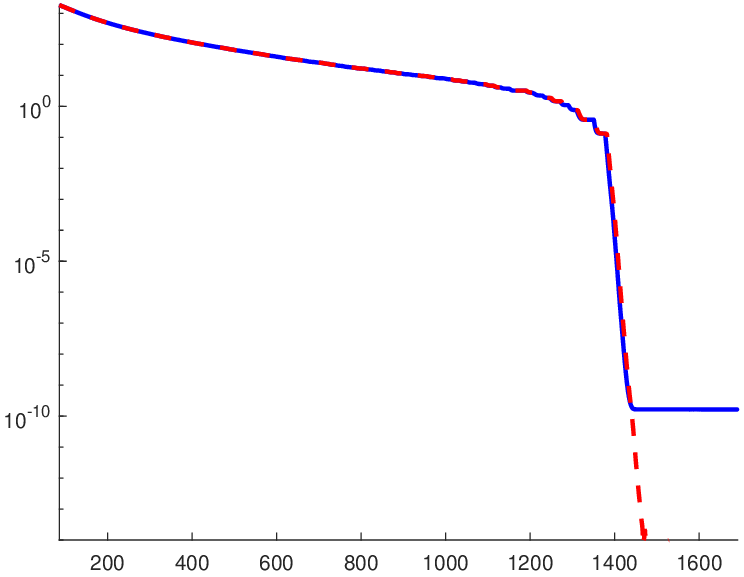}}
    
    \caption{Convergence history of matrices from the density functional theory calculation. The X-axis represents the number of matvecs, while the Y-axis shows the absolute error of Ritz values $\sum_{i=1}^{k}(\mu_{i}-\lambda_{i})$. The blue lines correspond to LC, while the red dashed lines represent KS.
    The number $k$ indicated in each subtitle in the figure is the number of occupied states. The maximum dimension of the Krylov subspace for both KS and LC is set to $m = \max\{80, 4k\}$.}
    \label{figexpDFT}
\end{figure}

\begin{figure}[!htbp]
    \centering
    \includegraphics[width=\figsizeD]{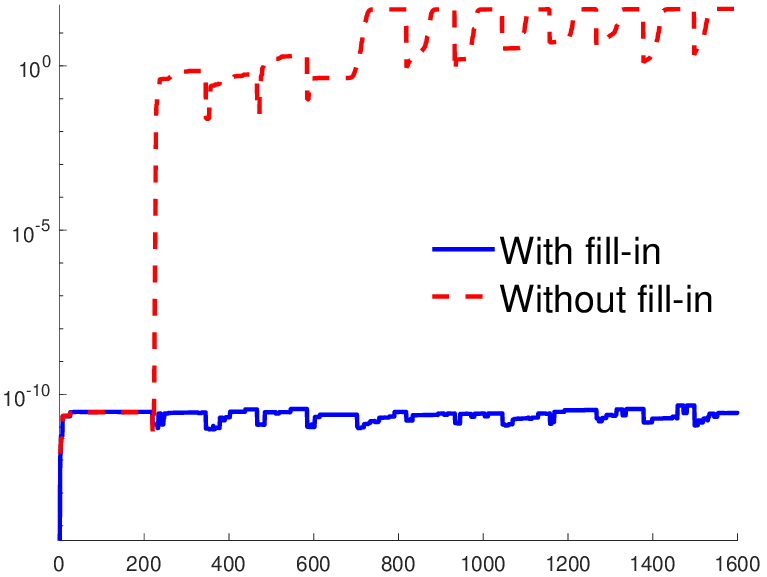}
    \includegraphics[width=\figsizeD]{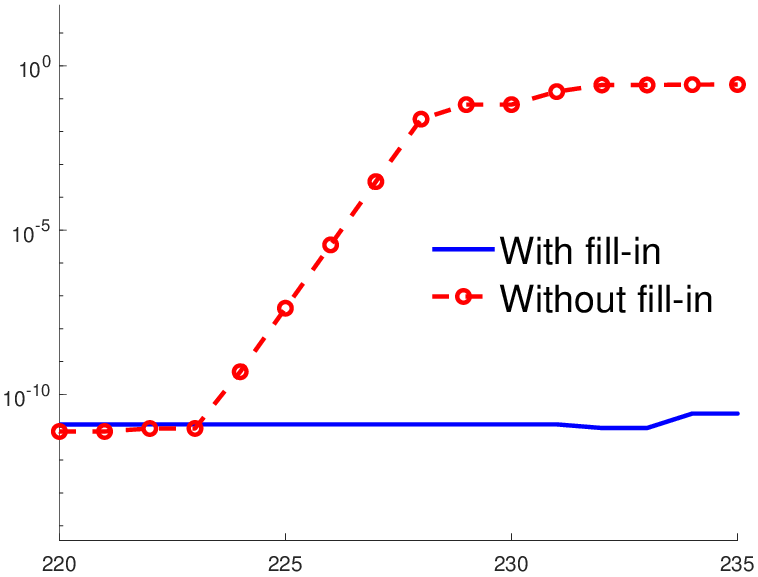}
    \caption{$\norm{Q_{j}^{\Ttran}AQ_{j}-T_{j}}$ \vs\ $\#$matvecs. Right panel is a zoomed-in view for the behavior after the first compression. The test matrix is \texttt{Ga10As10H30}, with the smallest $55$ eigenvalues being of interest. The maximum dimension of the Krylov subspace before compression is set to $220$.}
    \label{figexpDFTreorth}
\end{figure}

In the Lanczos method and KS, backward stability requires that $Q_{m}$ has nearly orthonormal columns.
An additional requirement on the orthogonality between $Q_{m+1}$ and $F_{m}$ arises in the Krylov-like decomposition \cref{eq:LCD}.  
To address this, we propose a new reorthogonalization scheme with fill-in in \cref{sec:reorth}.  
The following numerical experiment demonstrates that LC is not backward stable when using the classical Gram--Schmidt process with reorthogonalization (CGS2) alone.  
Using the matrix \texttt{Ga10As10H30} as a test case, we report the error $\norm{Q_{j}^{\Ttran}AQ_{j}-T_{j}}$ for reorthogonalization with and without fill-in in \cref{figexpDFTreorth}. It is evident that the scheme presented in \cref{sec:reorth} effectively preserves backward stability, whereas CGS2 loses accuracy immediately after the first compression.

\section{Conclusion}

In this paper, we have proposed a novel strategy for compressing the Krylov subspace in symmetric eigenvalue problems. Theoretical analysis shows that the proposed compression introduced only negligible error relative to the standard Lanczos method, while achieving a more favorable (re)orthogonalization complexity. Numerical experiments show its potential of outperforming the standard implicit restarting strategy. An important direction for future work is to address and overcome the stagnation observed in the numerical experiments.

\section*{Acknowledgments}

While working on the manuscript, the first author was a member of the Gran Sasso Science Institute, L'Aquila, and 
the research group INdAM-GNCS. The first author has been supported by 
the National Research Project (PRIN) ``FIN4GEO: Forward and Inverse Numerical Modeling of hydrothermal systems in volcanic regions with application to geothermal energy exploitation''.
\edit{The authors thank the referees for helpful comments and insights.}

\bibliographystyle{siamplain}

\end{document}